\documentclass[11pt]{article}
\usepackage{ae}
\usepackage{graphicx}
\usepackage{amsmath}
\usepackage{amssymb}
\usepackage{amsthm}
\usepackage{amsfonts}
\usepackage{latexsym}
\usepackage{verbatim}
\usepackage{verbatim}
\usepackage{url}
\usepackage[normalem]{ulem}
\newtheorem{thrm}{Theorem}[section]

\newtheorem{lmm}[thrm]{Lemma}
\newtheorem{crllr}[thrm]{Corollary}
\newtheorem{prpstn}[thrm]{Proposition}

\theoremstyle{definition}
\newtheorem{dfntn}[thrm]{Definition}
\newtheorem{rmrk}[thrm]{Remark}
\newtheorem{xmpl}[thrm]{Example}
\newtheorem{fact}[thrm]{Fact}
\newtheorem*{notation}{Notation}

\theoremstyle{remark}
\newtheorem*{note}{Note}

\newcommand{\cB}{{\mathcal B}}

\newcommand{\cA}{{\mathcal A}}

\newcommand{\cQ}{{\mathcal Q}}

\newcommand{\ZZ}{{\mathbb Z}}

\newcommand{\NN}{{\mathbb N}}

\newcommand{\cAstar}{\mathcal{A}^*}
\newcommand{\cAplus}{\mathcal{A}^+}
\newcommand{\empt}{\varepsilon}
\newcommand{\cAw}{\mathcal{A}^{\omega}}
\newcommand{\cAinf}{\mathcal{A}^{\infty}}

\newcommand{\rev}{\widetilde}

\newcommand{\bu}{\mathbf{u}}

\newcommand{\bw}{\mathbf{w}}
\newcommand{\bx}{\mathbf{x}}
\newcommand{\by}{\mathbf{y}}
\newcommand{\bz}{\mathbf{z}}

\newcommand{\bs}{\mathbf{s}}
\newcommand{\bt}{\mathbf{t}}

\newcommand{\cL}{\mathcal{L}}

\newcommand{\Ult}{\mbox{Ult}}
\newcommand{\eUlt}{\mbox{\emph{Ult}}}

\newcommand{\Alphit}{\mbox{{\em Alph}}}
\def\Alph{\textrm{Alph}}
\newcommand{\cR}{\mathcal{R}}

\newcommand{\br}{\mathbf{r}}

\def\TT{\mathrm{T}}

\numberwithin{equation}{section}

\author{Amy Glen\footnotemark[2] \and
Florence Lev\'e\footnotemark[3] \and Gw\'ena\"el Richomme\footnotemark[3]
}
\title{Quasiperiodic and Lyndon episturmian words\footnote{This work combines and extends two conference papers, one by the first author \cite{aG07orde} and the other by the second two authors \cite{fLgR07quasB}, presented at the {Sixth International Conference on Words}, Marseille, France, September 17--21, 2007.}}

\date{Submitted: May 7, 2008; Revised: September 14, 2008}

\begin{document}

\normalsize

\maketitle 

\footnotetext[2]{
\small Amy Glen\\
\small LaCIM, Universit\'e du Qu\'ebec \`a Montr\'eal, C.P. 8888, succursale Centre-ville,
Montr\'eal, Qu\'ebec, H3C 3P8, CANADA ~$\backslash$~ The Mathematics Institute, Reykjavik University, Kringlan 1, IS-103 Reykjavik, ICELAND \\
E-mail: \texttt{amy.glen@gmail.com} 
}
\footnotetext[3]{
\small F. Levé, G. Richomme \\ 
Universit\'e de Picardie Jules Verne, Laboratoire MIS (Modélisation, Information, Systèmes), 33 Rue Saint Leu, F-80039 Amiens cedex 1, FRANCE \\
E-mail: \texttt{\{florence.leve, gwenael.richomme\}@u-picardie.fr}
}

\hrule
\begin{abstract}
Recently the second two authors characterized {\em quasiperiodic} Sturmian words, proving that a Sturmian word is non-quasiperiodic if and only if it is an {\em infinite Lyndon word}. Here we extend this study to {\em episturmian words} (a natural generalization of Sturmian words) by describing all the quasiperiods of an episturmian word, which yields a characterization of quasiperiodic episturmian words in terms of their  {\em directive words}.  
Even further, we establish a complete characterization of all episturmian words that are Lyndon words. 
Our main results show that, unlike the Sturmian case, there is a much wider class of episturmian words that are non-quasiperiodic, besides those that are infinite Lyndon words. Our key tools are morphisms and directive words,  in particular {\em normalized} directive words, which we introduced in an earlier paper. Also of importance is the use of {\em return words} to characterize quasiperiodic episturmian words, since such a method could be useful in other contexts.
\medskip

\noindent {\bf Keywords}: episturmian word; Sturmian word; Arnoux-Rauzy sequence; episturmian morphism; lexicographic order; infinite Lyndon word; quasiperiodicity.
\vspace{0.1cm} \\
MSC (2000): 68R15.
\end{abstract}
\hrule

\section{Introduction}

{\em Sturmian words} are a fascinating family of infinite words defined on a 2-letter alphabet which have been extensively studied since the pioneering work of Morse and Hedlund in 1940 (see~\cite{MH1940}). Over the years, these infinite words have been shown to have numerous equivalent definitions and characterizations, and their beautiful properties are related to many fields like Number Theory, Geometry, Dynamical Systems, and Combinatorics on Words (see \cite{AS2003,mL02alge,nP02subs,jB07stur} for recent surveys). 

Many recent works have been devoted to generalizations of Sturmian words to arbitrary finite alphabets. An especially  interesting generalization is the  family of {\em episturmian words}, introduced by Droubay, Justin, and Pirillo in 2001 \cite{xDjJgP01epis} (see also \cite{jJgP02epis, jJgP04epis} for example). Episturmian words include not only the Sturmian words, but also the well-known {\em Arnoux-Rauzy sequences} (e.g., see \cite{pAgR91repr, jJgP02onac, nP02subs, rRlZ00agen}). More precisely, the family of episturmian words is composed of  the Arnoux-Rauzy sequences, images of the Arnoux-Rauzy sequences by {\em episturmian morphisms}, and certain periodic infinite words. In the binary case, Arnoux-Rauzy sequences are exactly the Sturmian words whereas episturmian words include all recurrent {\em balanced} words, that is, periodic balanced words and Sturmian words (see \cite{aGjJgP06char,PV2006,gR07aloc} for recent results relating episturmian words to the balanced property). See also \cite{aGjJ07epis} for a recent survey on episturmian theory.

Episturmian morphisms play a central role in the study of episturmian words (Section \ref{SS:EpiMorphisms} recalls the definition of these morphisms). Introduced first as a generalization of Sturmian morphisms, Justin and Pirillo \cite{jJgP02epis} showed that they are exactly the morphisms that preserve the aperiodic episturmian words. They also proved that any episturmian word is the image of another episturmian word by some so-called {\em pure episturmian morphism}. Even more, any episturmian word can be infinitely decomposed over the set of pure episturmian morphisms.  This last property allows  an episturmian word to be defined by one of its morphic decompositions or, equivalently, by a certain {\em directive word}, which is an infinite sequence of rules  for decomposing the given episturmian word by morphisms. In consequence, many properties of episturmian words can be deduced from properties of episturmian morphisms. This approach is used for instance in \cite{vBcHlZ06init, aG06acha, fLgR04quas,gR07conj,gR07aloc,rRlZ00agen} and of course in the papers of Justin {\em et al.}.

Morphic decompositions of Sturmian words have been used in \cite{fLgR07quas} to characterize quasiperiodic Sturmian words. Quasiperiodicity of finite words was first introduced by Apostolico and Ehrenfeucht \cite{aAaE93effi} in the following way: ``a word $w$ is quasiperiodic if there exists a second word $u \neq w$ such that every position of $w$ falls within some occurrence of $u$ in $w$''. The word $u$ is then called a quasiperiod of $w$. In the last fifteen years, quasiperiodicity and covering of finite words has been extensively studied (see \cite{aAmC01stri, IM1999} for some surveys). In \cite{sM04quas}, Marcus extended this notion to infinite words and opened some questions, particularly concerning quasiperiodicity of Sturmian words. After a brief answer to some of these questions in \cite{fLgR04quas}, the Sturmian case was fully studied in \cite{fLgR07quas} where it was proved that a Sturmian word is non-quasiperiodic if and only if it is an infinite Lyndon word. Here we extend this study to episturmian words.

In Sections~\ref{Ss:episturmian}--\ref{S:directiveWords}, we recall useful results on episturmian words and their directive words. A particularly important tool is the {\em normalization} of words directing the same episturmian word, which we recently introduced in \cite{fLgR07quasB, aGfLgR08dire}. This idea allows an episturmian word to be defined uniquely by its so-called {\em normalized directive word}, defined by some factor avoidance.
It can be seen as a generalization of a previous result by Berth\'e, Holton, and Zamboni \cite{vBcHlZ06init}, which was used  in \cite{fLgR07quas} to show that the directive word of a non-quasiperiodic Sturmian word can take only two possible (similar) forms. For non-binary  episturmian words, even those defined on a ternary alphabet, this simplicity does not hold since a combinatorial explosion of the number of cases occurs.  In particular there exist ternary non-quasiperiodic episturmian words that have infinitely many directive words. As such, the method we use to characterize quasiperiodic episturmian words greatly differs from the one in the Sturmian case.

In Section~\ref{S:quasiperiodicity}, we characterize quasiperiodic episturmian words. To prove it, we introduce a new way to tackle quasiperiodicity by stating an equivalent definition that is related to the notion of {\em return words}. From this, we show that any {\em standard} episturmian word (or {\em epistandard word}) is quasiperiodic; in particular, sufficiently long palindromic prefixes of an epistandard word are {\em quasiperiods} of it. We then extend this result by describing the quasiperiods of any (quasiperiodic) episturmian word (see Theorem~\ref{T:episturmianQuasiperiods}). This yields a characterization of  quasiperiodic episturmian words in terms of their directive words (Theorem~\ref{T:quasi-characterization}). Note that the set of quasiperiods of an episturmian word was previously described only for the Fibonacci word \cite{fLgR04quas}. (In \cite{fLgR07quas} this set was not described for quasiperiodic Sturmian words). At the end of Section~\ref{S:quasiperiodicity}, using the normalization aspect, we give a second characterization of quasiperiodic episturmian words which itself provides an effective way to decide whether or not a given episturmian word is quasiperiodic (see Theorem~\ref{T:normalized-quasi-characterization}).

Section~\ref{S:QuasiperiodicEpisturmianMorphisms} is concerned with the study of the action of episturmian morphisms in relation to quasiperiodicity. This study leads to non-trivial extensions of results in \cite{fLgR07quas}. Using this approach, we provide a completely different proof of our main characterization of quasiperiodic episturmian words. We also characterize episturmian morphisms that map any word onto a quasiperiodic one (see Section~\ref{SS:stongly-quasi-morphisms}). This result naturally allows us  to consider quasiperiodicity of words defined using episturmian morphisms.

Lastly, in Section~\ref{S:episturmianLyndon}, we characterize episturmian Lyndon words in terms of their directive words. This result shows that, unlike the Sturmian case, there exist non-quasiperiodic episturmian words that are not infinite Lyndon words.

\section{Episturmian words and morphisms\label{Ss:episturmian}}

We assume the reader is familiar with combinatorics on words and morphisms (e.g., see \cite{mL83comb,mL02alge}). In this section, we recall some basic definitions and properties relating to episturmian words which are  needed throughout the paper.  For the most part, we follow the notation and terminology of \cite{xDjJgP01epis, jJgP02epis, jJgP04epis, aGjJgP06char}.

\subsection{Notation and terminology}

Let $\cA$ denote a finite non-empty {\em alphabet}. A finite \emph{word} over $\cA$ is a finite sequence of letters from $\cA$. The {\em empty word} $\empt$ is the empty sequence. Under the operation of concatenation, the set $\cA^*$ of all finite words over $\cA$ is a {\em free monoid} with identity element $\empt$ and set of generators $\cA$. The set of {\em non-empty} words over $\cA$ is the {\em free semigroup} $\cA^+ = \cAstar \setminus \{\empt\}$.  

Given a finite word $w = x_{1}x_{2}\cdots x_{m} \in \cAplus$ with each $x_{i} \in \cA$, the \emph{length} of $w$, denoted by $|w|$, is equal to $m$. By convention, the empty word is the unique word of length $0$. We denote by $|w|_a$ the number of occurrences of the letter $a$ in the word $w$. If $|w|_a = 0$, then $w$ is said to be {\em $a$-free}.   
The \emph{reversal} of $w$, denoted by $\rev{w}$, is its mirror image: $\rev{w} = x_{m}x_{m-1}\cdots x_{1}$, and if $w = \rev{w}$, then $w$ is called a \emph{palindrome}. 

A (right) \emph{infinite word} (or simply \emph{sequence}) $\bx$ is a sequence indexed by $\NN^+$ with values in $\cA$, i.e., $\bx = x_1x_2x_3\cdots$ with each $x_i \in \cA$. The set of 
all infinite words over $\cA$ is denoted by $\cAw$.   An {\em ultimately periodic} infinite word can be written as $uv^\omega = uvvv\cdots$, for some $u$, $v \in \cAstar$, $v\ne \empt$. If $u = \empt$, then such a word is (purely) {\em periodic}. An infinite word that is not ultimately periodic is said to be {\em aperiodic}. For easier reading, infinite words are hereafter typically typed in boldface to distinguish them from finite words.

Given a set $X$ of words, $X^*$ (resp.~$X^\omega$) is the set of all finite (resp.~infinite) words that can be obtained by 
concatenating words of $X$. The empty word $\varepsilon$ belongs to $X^*$.

A finite word $w$ is a \emph{factor} of a finite or infinite word $z$ if $z = uwv$ for some words $u$, $v$ (where $v$ is infinite iff $z$ is infinite). In the special case $u = \empt$ (resp.~$v =\empt$),  we call $w$ a \emph{prefix} (resp.~\emph{suffix}) of $z$.  We use the notation $p^{-1}w$ (resp.~$ws^{-1}$) to indicate the removal of a prefix $p$ (resp.~suffix $s$) of the word $w$. An infinite word $\bx \in \cAw$ is called a \emph{suffix} of $\bz \in \cAw$ if there exists a word $w \in \cA^*$ such that $\bz = w\bx$. That is, $\bx$ is a shift of $\bz$, given by $\bx = \TT^{|w|}(\bz) = w^{-1}\bz$, where $\TT$ denotes the {\em shift map}: $\TT((x_n)_{n\geq1}) = (x_{n+1})_{n\geq1}$.  Note that a prefix or suffix $u$ of a finite or infinite word $w$ is said to be {\em proper} if $u \ne w$. For finite words $w \in \cA^*$, the shift map $\TT$ acts circularly, i.e., if $w = xv$ where $x \in \cA$, then $\TT(w) = vx$.

The \emph{alphabet} of a finite or infinite word $w$, denoted by $\Alph(w)$ is the set of letters occurring in $w$, and if $w$ is infinite, we denote by Ult$(w)$ the set of
all letters occurring infinitely often in $w$.  

A factor of an infinite word $\bx$ is \emph{recurrent} in $\bx$ if it occurs infinitely often in $\bx$, and $\bx$ itself is said to be \emph{recurrent} if all of its factors are recurrent in it. Furthermore, $\bx$ is \emph{uniformly recurrent} if for each $n$ there exists a positive integer $K(n)$ such that any factor of $\bx$ of length at least $K(n)$ contains all factors of $\bx$ of length $n$. Equivalently, $\bx$ is uniformly recurrent if any factor of $\bx$ occurs infinitely many times in $\bx$ with bounded gaps \cite{eCgH73sequ}.

\subsection{Episturmian words} \label{SS:EpiWords}

In this paper, our vision of episturmian words will be the characteristic property stated in Theorem~\ref{T:episturmian} (below). However, we first give one of their equivalent definitions to aid in understanding. For this, we recall that a factor $u$ of a finite or infinite word $w \in \cAinf$ is \emph{right} (resp.~\emph{left}) \emph{special} if $ua$, $ub$ (resp.~$au$, $bu$) are factors of $w$ for some letters $a$, $b \in \cA$, $a \ne b$.

An infinite word $\bt \in \cAw$ is \emph{episturmian} if its set of factors is closed under reversal and $\bt$ has at most one right (or equivalently left) special factor of each length. Moreover, an episturmian word is \emph{standard} if all of its left special factors are prefixes of it.

In the initiating paper \cite{xDjJgP01epis}, episturmian words were defined as an extension of standard episturmian words, which were themselves first introduced and studied as a generalization of standard Sturmian words using {\em palindromic closure} (see Theorem~\ref{T:epistandard} later). Specifically, an infinite word was said to be episturmian if it has exactly the same set of factors as some standard episturmian word \cite{xDjJgP01epis}. This definition is equivalent to the aforementioned one by Theorem 5 in~\cite{xDjJgP01epis}. Moreover, it was proved in \cite{xDjJgP01epis} that episturmian words are uniformly recurrent. Hence ultimately periodic episturmian words are (purely) periodic.

\begin{note} Hereafter, we refer to a standard episturmian word as an {\em epistandard word}, for simplicity.
\end{note}

To study episturmian words, Justin and Pirillo \cite{jJgP02epis} introduced {\em episturmian morphisms}. 
In particular they proved that these morphisms, which we recall below, are precisely the morphisms that preserve the set of aperiodic episturmian words. 

\subsection{Episturmian morphisms} \label{SS:EpiMorphisms}

\medskip

Let us recall that given an alphabet $\cA$, a \textit{morphism} $f$ on $\cA$ is a map from $\cA^*$ to $\cA^*$ such that $f(uv) = f(u) f(v)$ for any words $u$, $v$ over $\cA$.  A morphism on $\cA$ is entirely defined by the images of letters in $\cA$. All morphisms considered in this paper will be non-erasing: the image of any non-empty word is never empty. Hence the action of a morphism $f$ on $\cA^*$ can be naturally extended to infinite words; that is, if $\bx = x_1x_2x_3 \cdots \in \cAw$, then $f(\bx) = f(x_1)f(x_2)f(x_3)\cdots$. 

In what follows, we will denote the composition of morphisms by juxtaposition as for concatenation of words. 

\medskip
 
Episturmian morphisms are the compositions of the permutation morphisms (i.e., the morphisms $f$ such that $f(\cA) = \cA$) and the morphisms $L_a$ and $R_a$ where, for all $a \in \cA$:
\[
  L_a: \left\{\begin{array}{lll}
               a &\mapsto &a \\
               b &\mapsto &ab 
               \end{array}\right. , 
               \quad R_a: \left\{\begin{array}{lll}
               a &\mapsto &a \\
               b &\mapsto &ba    
               \end{array}\right. \quad \mbox{for all $b \ne a$ in $\cA$}.
\] 
Here we will work only on {\it pure} episturmian morphisms, i.e., morphisms obtained by composition of elements of  the sets:
\[
\mathcal{L}_\cA=\{L_a \mid a \in \cA\} \quad \mbox{and} \quad \mathcal{R}_\cA =\{R_a \mid a \in \cA\}.
\]

\begin{note}
In \cite{jJgP02epis}, the morphism $L_a$ (resp.~$R_a$) is denoted by $\psi_a$ (resp.~$\bar\psi_a$). We adopt the current notation to emphasize the action of $L_a$ (resp.~$R_a$) when applied to a word, which consists of placing an occurrence of the letter $a$ on the left (resp.~right) of each occurrence of any letter different from~$a$.
\end{note}

{\em Epistandard morphisms} (resp.~{\em pure episturmian morphisms}, {\em pure epistandard morphisms}) are the morphisms obtained by concatenation of morphisms in $\mathcal{L}_\cA$ and permutations on $\cA$ (resp.~in
$\mathcal{L}_\cA \cup \mathcal{R}_\cA$, in $\mathcal{L}_\cA$).
Note that the episturmian morphisms are exactly the {\em Sturmian morphisms} when $\cA$ is a $2$-letter alphabet.

\subsection{Morphic decomposition of episturmian words} \label{SS:morphicDecomp}

Justin and Pirillo \cite{jJgP02epis} proved the following insightful characterizations of epistandard and episturmian words (see Theorem \ref{T:episturmian} below), which show that any episturmian word can be {\em infinitely decomposed} over the set of pure episturmian morphisms.

The statement of Theorem~\ref{T:episturmian} needs some extra definitions and notation. First we define the following new alphabet, $\bar{\cA} = \{\bar{x} \mid x \in \cA \}$. A letter $\bar x$ is considered to be $x$ with {\em spin} $R$,  whilst $x$ itself has spin $L$. 
A finite or infinite word over 
$\cA \cup \bar{\cA}$ is called a \textit{spinned} word.
To ease the reading, we sometimes call a letter with spin $L$ (resp.~spin $R$) an $L$-spinned (resp.~$R$-spinned) letter. By extension, an $L$-spinned (resp.~$R$-spinned) word is a word having only letters with spin $L$ (resp.~spin~$R$).

The {\em opposite} $\bar{w}$ of a finite or infinite spinned word $w$ is obtained from $w$ by exchanging all spins in $w$. For instance, if $v=ab \bar a$, then $\bar v=\bar a \bar b a$. When $v \in \cA^+$, then its opposite $\bar v \in \bar\cA^+$ is an $R$-spinned word and we set $\bar\empt = \empt$. Note that, given a finite or infinite word $w = w_1w_2\cdots$ over $\cA$, we sometimes denote $\breve w = {\breve w}_1{\breve w}_2\cdots$ any spinned word such that $\breve w_i = w_i$ if $\breve w_i$ has spin $L$ and $\breve w_i = \bar w_i$ if $\breve w_i$ has spin $R$. Such a word $\breve w$ is called a {\em spinned version} of $w$.

\medskip
\begin{note} In Justin and Pirillo's original papers, spins are 0 and 1 instead of $L$ and $R$. It is convenient here to change this vision of the spins because of the relationship with episturmian morphisms, which we now recall. 
\end{note}

For  $a \in \cA$, let $\mu_a = L_a$ and $\mu_{\bar a} = R_a$. This operator $\mu$ can be naturally extended (as done in~\cite{jJgP02epis}) to a morphism mapping any word over $(\cA \cup \bar\cA)$ into a pure episturmian morphism: for a spinned finite word $\breve w = \breve w_1\cdots \breve w_n$ over $\cA \cup \bar\cA$, $\mu_{\breve w} = \mu_{\breve w_1}\cdots \mu_{\breve w_n}$ ($\mu_\varepsilon$ is the identity morphism). We will say that the word $w$ {\em directs} or is a {\em directive word} of the morphism $\mu_w$. The following result extends the notion of directive words to infinite episturmian words.

\begin{thrm} {\rm \cite{jJgP02epis}} 
\label{T:episturmian}

\begin{enumerate} 
\item[$i)$] An infinite word $\bs \in \cAw$ is epistandard  if and only if there exists an infinite word  $\Delta = x_1x_2x_3\cdots$ over $\cA$ and an infinite sequence $(\bs^{(n)})_{n \geq 0}$ of infinite words such that $\bs^{(0)} = \bs$ and for all $n \geq 1$, $\bs^{(n-1)} = L_{x_n}(\bs^{(n)})$.

\item[$ii)$] An infinite word $\bt \in \cAw$ is episturmian if and only if there exists a spinned infinite word $\breve\Delta = \breve x_1 \breve x_2 \breve x_3 \cdots$ over $\cA \cup \bar\cA$ and an infinite sequence $(\bt^{(n)})_{n \geq 0}$ of recurrent infinite words such that    
$\bt^{(0)} = \bt$ and for all $n \geq 1$, $\bt^{(n-1)} = \mu_{\breve x_n}(\bt^{(n)})$. 
\end{enumerate}
\end{thrm}

For any epistandard word (resp.~episturmian word) $\bt$ and $L$-spinned (resp.~spinned) infinite word $\Delta$ (resp.~$\breve\Delta$) satisfying the conditions of the above theorem, we say that $\Delta$ (resp.~$\breve\Delta$) is a {\em (spinned) directive word} for $\bt$ or that $\bt$ is {\em directed by} $\Delta$ (resp.~$\breve\Delta$). 

\begin{rmrk} \label{R:directive} It follows immediately from Theorem~\ref{T:episturmian} that if $\bt$ is an episturmian word directed by a spinned infinite word $\breve\Delta$, then each $\bt^{(n)}$ (as defined in part $ii)$) is an episturmian word directed by $\TT^n(\breve\Delta) = \breve x_{n+1}\breve x_{n+2} \breve x_{n+3} \cdots$.
\end{rmrk}

By Theorem~1 in \cite{xDjJgP01epis} (see also Theorem~\ref{T:epistandard} later), any epistandard word has a unique $L$-spinned directive word, but also  has infinitely many other directive words (see \cite{jJgP02epis,jJgP04epis, aGfLgR08dire}). For example, the {\em Tribonacci word} (or {\em Rauzy word}~\cite{gR82nomb}) is  directed by $(abc)^\omega$ and also by $(abc)^n\bar{a}\bar{b}\bar{c}(a\bar{b}\bar{c})^\omega$ for each $n\geq 0$, as well as infinitely many other spinned words. More generally, by Proposition~3.11 in \cite{jJgP02epis}, any spinned infinite word $\breve\Delta$  having infinitely many $L$-spinned letters directs a unique episturmian word $\bt$ beginning with the left-most $L$-spinned letter in $\breve\Delta$. Moreover, by one of the main results in \cite{aGfLgR08dire} (see Theorem~\ref{T:directSame} later), $\bt$ has infinitely many other directive words.

The following important fact links the two parts of Theorem~\ref{T:episturmian}.

\begin{fact} \label{F:episturmian} {\rm \cite{jJgP02epis}} If $\bt$ is an episturmian word directed by a spinned version $\breve \Delta$ of an $L$-spinned infinite word $\Delta$, then $\bt$ has exactly the same set of factors as the (unique) epistandard word $\bs$ directed by $\Delta$.
\end{fact}

Moreover, with the same notation as in the above remark, the episturmian word $\bt$ is periodic if and only if the epistandard word $\bs$ is periodic, and this holds if and only if $|\Ult(\Delta)| = 1$  (see \cite[Prop.~2.9]{jJgP02epis}). More precisely, a periodic episturmian word takes the form $(\mu_{\breve w}(x))^\omega$ for some finite spinned word $\breve w$ and letter $x$.

\begin{note} Sturmian words are precisely the aperiodic episturmian words on a 2-letter alphabet.
\end{note}

When an episturmian word is aperiodic, we have the following fundamental link between the words $(\bt^{(n)})_{n \geq 0}$ and the spinned infinite word $\breve\Delta$ occurring in Theorem~\ref{T:episturmian}: if $a_n$ is the first letter of $\bt^{(n)}$, then $\mu_{\breve x_1 \cdots \breve x_n}(a_n)$ is a prefix of $\bt$ and the sequence $(\mu_{\breve x_1 \cdots \breve x_n}(a_n))_{n \geq 1}$ is not ultimately constant (since $\breve\Delta$ is not ultimately constant), then $\bt = \lim_{n \rightarrow \infty} \mu_{\breve x_1 \cdots \breve x_n}(a_n)$. This fact is a slight generalization of a result of Risley and Zamboni \cite[Prop.~III.7]{rRlZ00agen} on {\em S-adic representations} for characteristic Arnoux-Rauzy sequences.  See also the recent paper \cite{vBcHlZ06init} for S-adic representations of Sturmian words. Note that {\em $S$-adic dynamical systems} were introduced by Ferenczi \cite{sF99comp} as {\em minimal dynamical systems}  (e.g., see \cite{nP02subs})  generated by a finite number of substitutions. In the case of episturmian words, the notion itself is actually a reformulation of the well-known {\em Rauzy rules}, as studied in \cite{gR85mots}. In fact, it is well-known that the {\em subshift} of an aperiodic episturmian word $\bt$ (i.e., the topological closure of the shift orbit of $\bt$)  is a {minimal dynamical system}, i.e., it  consists of all  the episturmian words with the same set of factors as $\bt$.

\section{Useful results on directive words} \label{S:directiveWords}

Notions concerning directive words of episturmian words and morphisms have been recalled in the previous section. Two natural questions concerning these words are:  When do two distinct finite spinned words direct the same episturmian morphism?  When do two distinct spinned infinite words direct the same unique episturmian word? In this section we recall existing answers to these questions.  We also present a way to uniquely define any episturmian word through a {\em normalization} of its directive words. This powerful tool was recently introduced in our papers \cite{aGfLgR08dire, fLgR07quasB}. 

\subsection{Presentation versus block-equivalence}

Generalizing a study of the monoid of Sturmian morphisms by Séébold \cite{See1991}, the third author \cite{gR03conj} answered the question: ``When do two distinct finite spinned words direct the same episturmian morphism?'' by giving a presentation of the monoid of episturmian morphisms. This result was reformulated in \cite{gR03lynd} using another set of generators and it was independently and differently treated in \cite{jJgP04epis}. As a direct consequence, one can see that the monoid of pure epistandard morphisms is a free monoid and one can obtain the following presentation of the monoid of pure episturmian morphisms:

\begin{thrm}\label{T:presentation}{\rm (direct consequence of \cite[Prop.~6.5]{gR03lynd}; reformulation of \cite[Th.~2.2]{jJgP04epis})}

The monoid of pure episturmian morphisms with $\{L_{\alpha}, R_{\alpha} \mid \alpha\in \cA\}$ as a set of generators has the following presentation:
$$R_{a_1} R_{a_2} \cdots R_{a_k} L_{a_1} = L_{a_1} L_{a_2} \cdots L_{a_k} R_{a_1}$$
where $k \geq 1$ is an integer and $a_1, \ldots ,a_k \in \cA$ with $a_1 \neq a_i$ for all $i$, $2 \leq i \leq k$. 
\end{thrm}

This result means that two different compositions of morphisms in ${\cal L}_{\cal A} \cup  {\cal R}_{\cal A}$ yield the same pure episturmian morphism if and only if one composition can be deduced from the other in a rewriting system, called the {\em block-equivalence} in \cite{jJgP04epis}. 
Although Theorem~\ref{T:presentation} allows us to show that many properties of episturmian words are linked to properties of episturmian morphisms, it will be convenient for us to have in mind the block-equivalence that we now recall.

A word of the form $xvx$, where $x \in \cA$ and $v \in (\cA\setminus\{x\})^*$, is called a ($x$-based) {\em block}. A ($x$-based) {\em block-transformation} is the replacement in a spinned word of an occurrence of $xv\bar x$ (where $xvx$ is a block) by $\bar x \bar v x$ or vice-versa. Two finite spinned words $w$, $w'$ are said to be {\em block-equivalent} if  we can pass from one to the other by a (possibly empty) chain of block-transformations, in which case we write $w \equiv w'$. For example, $\bar b\bar a b \bar c b \bar a \bar c$ and  $b a bc\bar b\bar a \bar c$ are block-equivalent because $\bar b\bar a b \bar c b \bar a \bar c \rightarrow ba\bar b \bar c b\bar a \bar c  \rightarrow b a bc\bar b\bar a \bar c$ and vice-versa. 
The block-equivalence is an equivalence relation over spinned words, and moreover one can observe that if $w \equiv w'$ then $w$ and $w'$ are spinned versions of the same word over $\cA$.

\medskip

Theorem~\ref{T:presentation} can be reformulated in terms of block-equivalence:

\medskip
\noindent
\textbf{Theorem~\ref{T:presentation}.} 
\textit{Let $w$, $w'$ be two spinned words over $\cA \cup \bar\cA$. Then $\mu_{w} = \mu_{w'}$ if and only if $w \equiv w'$. 
}

\subsection{Words directing the same episturmian word} \label{SS:directSame}

 Using the block-equivalence notion, the question: ``When do two distinct spinned infinite words direct the same unique episturmian word?'' was almost completely solved by Justin and Pirillo in \cite{jJgP04epis} for {\em bi-infinite episturmian words}, i.e., episturmian words with letters indexed by $\ZZ$ (and not by $\NN$ as we consider here). More recently, in \cite{aGfLgR08dire}, we showed that Justin and Pirillo's results on directive words of  bi-infinite episturmian words are still valid for words directing (right-infinite) episturmian words. We also established the following complete characterization of pairs of spinned infinite words directing the same unique episturmian word. Not only does our characterization provide the relative forms of two spinned infinite words directing the same episturmian word, but it also fully solves the periodic case, which was only partially solved in \cite{jJgP04epis}.

 \begin{thrm}\label{T:directSame} {\em  \cite{aGfLgR08dire}}
Given two spinned infinite words $\Delta_1$ and $\Delta_2$, the following assertions are equivalent.
\begin{description}
\item{i)} $\Delta_1$ and $\Delta_2$ direct the same right-infinite episturmian word;
\item{ii)} $\Delta_1$ and $\Delta_2$ direct the same bi-infinite episturmian word;
\item{iii)} One of the following cases holds for some $i, j$ such that $\{i, j\} = \{1, 2\}$:
\begin{enumerate}
\item \label{Ti:1} $\Delta_i = \prod_{n \geq 1} v_n$, $\Delta_j = \prod_{n \geq 1} z_n$ where $(v_n)_{n \geq 1}, (z_n)_{n \geq 1}$ are spinned words such that $\mu_{v_n} = \mu_{z_n}$ for all $n \geq 1$;

\item \label{Ti:2} $\Delta_i = {w} x \prod_{n \geq 1} v_n {\breve x}_n$, $\Delta_j = {w'} {\bar x} \prod_{n \geq 1} {\bar v}_n {\hat x}_n$ where ${w}$, ${w'}$ are spinned words such that $\mu_{w} = \mu_{w'}$, $x$ is an $L$-spinned letter, $(v_n)_{n \geq 1}$ is a sequence of non-empty $x$-free $L$-spinned words, and $({\breve x}_n)_{n \geq 1}$, $({\hat x}_n)_{n \geq 1}$ are sequences of non-empty spinned words over $\{x, \bar x\}$ such that, for all $n \geq 1$, $|{\breve x}_n| = |{\hat x}_n|$ and  $|{\breve x}_n|_x = |{\hat x}_n|_x$;

\item \label{Ti:4} $\Delta_1 = w \bx$ and $\Delta_2 = w'\by$ where $w$, $w'$ are spinned words and  $\bx \in \{x, \bar x\}^\omega$, $\by \in \{y, \bar y\}^\omega$ are spinned infinite words for some letters $x$, $y$ such that $\mu_{w}(x) = \mu_{w'}(y)$. 
\end{enumerate} 
\end{description}
 \end{thrm}

In items \ref{Ti:1} and \ref{Ti:2} of Theorem~\ref{T:directSame}, the two considered directive words are spinned versions of the same $L$-spinned word. This does not hold in item \ref{Ti:4}, which concerns only periodic episturmian words. In particular, we make the following observation:

\begin{fact} 
\label{F:aperiodic-spinned-versions}  If an {\em aperiodic} episturmian word is directed by two spinned words $\Delta_1$ and $\Delta_2$, then $\Delta_1$ and $\Delta_2$ are spinned versions of the same $L$-spinned word $\Delta$. 
\end{fact}

As an example of item \ref{Ti:4}, one can consider the periodic episturmian word $(bcba)^\omega$ which is directed by both $bca^\omega$ and $b \bar ac^\omega$. Note also that $(bcba)^\omega$ is epistandard and has the same set of factors as the epistandard word $(babc)^\omega$ directed by $bac^\omega$. Actually, in view of Fact~\ref{F:episturmian}, we observe the following:

\begin{fact} \label{F:subshift}  The subshift of any aperiodic episturmian word contains a unique (aperiodic) epistandard word, whereas the subshift of a periodic episturmian word contains exactly two (periodic) epistandard words, except if this word is $a^\omega$ with $a$ a letter. 
\end{fact}

\subsection{Normalized directive word of an episturmian word} \label{SS:normalized}

 Items \ref{Ti:2} and \ref{Ti:4} of Theorem~\ref{T:directSame} show that any episturmian word is directed by a spinned word having infinitely many $L$-spinned letters, but also by a spinned infinite word having both infinitely many $L$-spinned letters and infinitely many $R$-spinned letters. To emphasize the importance of these facts, let us recall from Proposition~3.11 in \cite{jJgP02epis} that if $\breve\Delta$ is a spinned infinite word over $\cA \cup \bar\cA$ with infinitely many $L$-spinned letters, then there exists a unique episturmian word $\bt$ directed by $\breve\Delta$. Unicity comes from the fact that the first letter of $\bt$ is fixed by the first $L$-spinned letter in $\breve\Delta$.

To work on Sturmian words, Berthé, Holton and Zamboni~\cite{vBcHlZ06init} proved that 
any Sturmian word has a unique directive word over $\{a,b,\bar a, \bar b\}$ containing infinitely many $L$-spinned letters but no factor of the form ${\bar a}{\bar b}^na$ or ${\bar b}{\bar a}^nb$ with $n$ an integer. 
Using Theorems~\ref{T:presentation} and \ref{T:directSame}, we recently generalized this result  to episturmian words:

\begin{thrm}\label{T:normalisation} {\em \cite{aGfLgR08dire, fLgR07quasB}}
Any episturmian word $\bt$ has a spinned directive word containing infinitely many $L$-spinned letters, but no factor in ${\bigcup}_{a\in\cA} \bar a \bar\cA^*a$. Such a directive word is unique if $\bt$ is aperiodic.
\end{thrm}

Note that unicity does not necessarily hold for periodic episturmian words. For example, the periodic episturmian word $(ab)^\omega = L_a(b^\omega) = R_b(a^\omega)$ is directed both by $a{b}^\omega$ and by $\bar{b}a^\omega$ ($L_a(b) = ab = R_b(a)$).

A directive word of an aperiodic episturmian word $\bt$ with the above property is called the {\em normalized directive word} of $\bt$. We extend this definition to morphisms: a finite spinned word $w$ is said to be a {\em normalized directive word} of the morphism $\mu_w$ if $w$ has no factor in ${\bigcup}_{a\in\cA} \bar a \bar\cA^*a$.

\medskip

One can observe from Theorem~\ref{T:presentation} that for any morphism in $L_a \mathcal{L}_\cA^\ast R_a$, we can find another decomposition of the morphism in the set $R_a \mathcal{R}_\cA^\ast L_a$. Equivalently, for any spinned word in $a\cA^*{\bar a}$, there exists a word $w'$ in ${\bar a}{\bar \cA}^*a$ such that $\mu_w = \mu_{w'}$. This was the main idea used in the proof of Theorem~\ref{T:normalisation}. 

\medskip

\begin{xmpl}
Let $f$ be the pure episturmian morphism with directive word
${\bar a}{\bar b}c{\bar b}a{\bar b}{\bar a}{\bar c}{\bar b} {\bar a}{\bar c}a$.  
By Theorem~\ref{T:presentation},  
$\mu_{{\bar a}{\bar c}{\bar b}{\bar a}{\bar c}a} = 
 \mu_{{\bar a}{\bar c}{\bar b}ac{\bar a}} = 
 \mu_{acb{\bar a}c{\bar a}}$ and hence 
$f= \mu_{{\bar a}{\bar b}c{\bar b}a{\bar b}{\bar a}{\bar c}{\bar b} {\bar a}{\bar c}a} = 
\mu_{{\bar a}{\bar b}c{\bar b} a {\bar b}acb{\bar a}c{\bar a}}$
and ${\bar a}{\bar b}c{\bar b}a{\bar b}acb{\bar a}c{\bar a}$ is the normalized directive word of $f$.
\end{xmpl}

\subsection{\label{SS:uniqueDirective}Episturmian words having a unique directive word}

Using our characterization of pairs of words directing the same episturmian word (Theorem~\ref{T:directSame}) together with normalization (Theorem~\ref{T:normalisation}), we recently characterized episturmian words having a unique directive word.

\begin{thrm} \label{T:uniqueDirective} {\em \cite{aGfLgR08dire}}
An episturmian word has a unique directive word if and only if its (normalized) directive word contains 
1) infinitely many $L$-spinned letters,
2) infinitely many $R$-spinned letters, 
3) no factor in ${\bigcup}_{a\in\cA} \bar a \bar\cA^*a$, 
4) no factor in ${\bigcup}_{a\in\cA} a \cA^*\bar a$. Such an episturmian word is necessarily aperiodic.
\end{thrm}

As an example, a particular family of episturmian words having unique directive words consists of those directed by {\em regular wavy words}, i.e., spinned infinite words having both infinitely many $L$-spinned letters and infinitely many $R$-spinned letters such that each letter occurs with the same spin everywhere in the directive word. More formally, a spinned version $\breve w$ of a finite or infinite word $w$ is said to be {\em regular} if, for each letter $x \in \Alph(w)$, all occurrences of $\breve x$ in $\breve w$ have the same spin $(L$ or $R)$. For example, $a\bar b a a \bar c \bar b$ and $(a\bar b c)^\omega$ are regular, whereas $a \bar b a \bar a \bar c b$ and $(a \bar b \bar a)^\omega$ are not regular. In \cite{jJgP04epis}, a spinned infinite word is said to be {\em wavy} if it contains infinitely many $L$-spinned letters and infinitely many $R$-spinned letters. For example, the two preceding infinite words are wavy.

In the Sturmian case, we have:

\begin{prpstn} \label{P:uniqueDirective-Sturmian} {\em \cite{aGfLgR08dire}}
Any Sturmian word has a unique spinned directive word or infinitely many spinned directive words. Moreover, a Sturmian word has a unique directive word if and only if its (normalized) directive word is regular wavy.
\end{prpstn}

In the next section, we shall see that any episturmian word having a unique directive word is necessarily non-quasiperiodic. This will follow from Theorem~\ref{T:uniqueDirective} and our characterizations of quasiperiodic episturmian words (Theorems~\ref{T:episturmianQuasiperiods}, \ref{T:quasi-characterization}, and \ref{T:normalized-quasi-characterization}).

\section{Quasiperiodicity of episturmian words} \label{S:quasiperiodicity}

\subsection{Quasiperiodicity}

Recall (from \cite{aAmFcI91opti, IM1999, sM04quas}) that a finite or infinite word $w$ is quasiperiodic if it can be constructed by concatenations and superpositions of one of its proper factors $u$, which is called a \textit{quasiperiod} of $w$ (or the \textit{smallest quasiperiod} of $w$ when it is of minimal length). We also say that $u$ covers $w$ or that $w$ is $u$-quasiperiodic. For example, the word $w=abaababaabaababaaba$ has $aba$, $abaaba$, $abaababaaba$ as quasiperiods, and the smallest quasiperiod of $w$ is $aba$. Words that are not quasiperiodic are naturally called {\em non-quasiperiodic} words.

When defining infinite quasiperiodic words, for convenience, we consider the words preceding the occurrences of a quasiperiod: an infinite word $\bw$ is
\textit{quasiperiodic} if and only if there exist a finite word $u$ and words
$(p_n)_{n\geq 0}$ such that $p_0 = \varepsilon$, 
$|p_n| < |p_{n+1}| \leq |p_nu|$, and $p_nu$ is a prefix of
$\bw$ for all $n\geq 0$. Then $u$ is a \textit{quasiperiod} of $\bw$ and we say that the sequence
$(p_nu)_{n\geq 0}$ is \textit{a covering sequence of prefixes of the
word $\bw$}. Necessarily, any quasiperiod of a quasiperiodic word must be a prefix of it.

Readers will find several examples of infinite quasiperiodic words in \cite{fLgR04quas,Mar2002,sM04quas}. Let us mention for instance that the {\em Fibonacci word}, directed by $(ab)^\omega$, is $aba$-quasiperiodic (see \cite{fLgR04quas}).

Let us now recall a simple, yet important, fact about quasiperiodic words.

\begin{fact}\label{rnew4}\label{fLgR07quas}
If $w$ is a (finite or infinite) $u$-quasiperiodic word and $f$ is a non-erasing morphism, then $f(w)$ is $f(u)$-quasiperiodic.
\end{fact}

Note that the converse of this fact is not true. For example, let $f=R_a$ and $\bw=ab^\omega$, then $R_a(\bw)=a(ba)^\omega$. The word $R_a(\bw)$ is covered by $R_a(ab)$, but $\bw$ is not covered by $ab$. 

\subsection{Return words and quasiperiodicity} \label{SS:return-words}

We now use the notion of a `return word' to give an equivalent definition of quasiperiodicity (see Lemma~\ref{L:quasi}), which proves to be a useful tool for studying quasiperiodicity in episturmian words.

Return words were introduced independently by Durand \cite{fD98acha} and by Holton and Zamboni \cite{cHlZ99desc} when studying primitive substitutive sequences. Such words can be defined in the following way.

\begin{dfntn} \label{D:return-word}
Let $v$ be a recurrent factor of an infinite word $\bw = w_1w_2w_3\cdots$, starting at positions $n_1 <n_2< n_3 \cdots$ in $\bw$. Then each word $r_i = w_{n_i}w_{n_i+1}\cdots w_{n_{i+1} -1}$ is called a {\em return to $v$} in $\bw$. 
\end{dfntn}

That is, a return to $v$ in $\bw$ is a non-empty factor of $\bw$ beginning at an occurrence of $v$ and ending exactly before the next occurrence of $v$ in $\bw$. Thus, if $r$ is a return to $v$ in $\bw$, then $rv$ is a factor of $\bw$ that contains exactly two occurrences of $v$, one as a prefix and one as a suffix.
As any episturmian word $\bt$ is uniformly recurrent \cite{xDjJgP01epis}, each factor of $\bt$ has only a finite number of different returns (for more details see Theorem~\ref{T:return-words}).

\begin{note} A return to $v$ in $\bw$ always has $v$ as a prefix or is a prefix of $v$. In particular, we observe that a return to $v$  is not necessarily longer than $v$, in which case $v$ has {\em overlapping} occurrences in $\bw$ (i.e., $vz^{-1}v$ is a factor of $\bw$ for some non-empty word $z$). 
We say that $v$ has {\em adjacent} occurrences in $\bw$ if $vv$ is a factor of $\bw$. In this case, if $v$ is {\em primitive} (i.e., not an integer power of a shorter word), then $v$ is a return to itself; otherwise, the corresponding return to $v$ is the {\em primitive root} of~$v$. 
\end{note}

In terms of return words, we have the following equivalent definition of a quasiperiodic infinite word.

\begin{lmm} \label{L:quasi}
A finite word $v$ is a quasiperiod of an infinite word $\bw$ if and only if $v$ is a recurrent prefix of $\bw$ such that any return to $v$ in $\bw$ has length at most $|v|$.
\end{lmm}

\begin{proof} If $v$ is a quasiperiod of $\bw$, then $v$ is a prefix of $\bw$ and its occurrences entirely cover $\bw$. That is, $v$ is recurrent in $\bw$ and successive occurrences of $v$ in $\bw$ are either adjacent or overlap, and hence any return to $v$ has length at most $|v|$.  Conversely, if $v$ is a recurrent prefix of $\bw$ such that any return to $v$ has length at most $|v|$, then successive occurrences of $v$ in $\bw$ are either adjacent or overlap, and hence entirely cover $\bw$. Thus $\bw$ is $v$-quasiperiodic. 
\end{proof}

Immediately:

\begin{crllr} \label{Cor:quasi} An infinite word $\bw$ is quasiperiodic if and only if there exists a recurrent prefix $v$ of $\bw$ such that any return to $v$ in $\bw$ has length at most $|v|$, in which case $v$ is a quasiperiod of $\bw$. Moreover, the shortest such prefix $v$ is the smallest quasiperiod of $\bw$.
\end{crllr}

A noteworthy fact is that a quasiperiodic infinite word is not necessarily recurrent \cite{sM04quas}, although it must have a prefix that is recurrent in it.

\subsection{\label{S:returnWordsEpisturmian}Return words and palindromic closure in episturmian words}

In this section we recall a characterization of return words in episturmian words, given by Justin and Vuillon in \cite{jJlV00retu}. For this, we first recall the construction of epistandard words using {\em palindromic right-closure} as well as some related properties from \cite{xDjJgP01epis, jJgP02epis} that will be used throughout Sections~\ref{SS:epistandard-quasiperiods}~to~\ref{SS:stongly-quasi-morphisms}.

The \emph{palindromic right-closure} $w^{(+)}$ of a finite word $w$ is the (unique) shortest palindrome having $w$ as a prefix (see \cite{aD97stur}). That is, $w^{(+)} = wv^{-1}\rev{w}$ where $v$ is the longest palindromic suffix of $w$. 
The {\em iterated palindromic closure} function \cite{jJ05epis}, denoted by $Pal$, is defined recursively as follows. Set $Pal(\empt) = \empt$ and, for any word $w$ and letter $x$, define $Pal(wx)~=~(Pal(w)x)^{(+)}$. For instance, $Pal(abc)=(Pal(ab)c)^{(+)} = (abac)^{(+)} = abacaba$.

Generalizing a construction given in \cite{aD97stur} for {\em standard Sturmian words}, Droubay, Justin and Pirillo established the following characterization of epistandard words.

\begin{thrm} \label{T:epistandard} {\em \cite{xDjJgP01epis}}
An infinite word $\bs \in \cAw$ is epistandard if and only if there exists an infinite word $\Delta = x_1x_2x_3\cdots$ ($x_i \in \cA$) such that $\bs = \lim_{n\rightarrow\infty} Pal(x_1 \cdots x_n)$.
\end{thrm}

Note that the palindromes $Pal(x_1 \cdots x_n)$ are very often denoted by $u_{n+1}$ in the literature.

In \cite{jJgP02epis}, Justin and Pirillo showed that the word $\Delta$ is exactly the directive word  of $\bs$ as it occurs in Theorem~\ref{T:episturmian}. Moreover, by construction, $\Delta$ {\em uniquely} determines the epistandard word~$\bs$.
Notice also that by construction, the words $(Pal(x_1 \cdots x_i))_{i \geq 0}$ are exactly the palindromic prefixes of $\bs$.

\medskip
There exist many relations between palindromes and episturmian morphisms. The following ones will be useful in the next few sections. 
First recall from \cite{jJ05epis, jJgP02epis} that we have
\begin{equation} \label{eq:formula(3)}
Pal(wv) = \mu_w(Pal (v))Pal(w) \quad \mbox{for any words $w$, $v$}.
\end{equation}
In particular, for any $x$ a letter, $Pal(xv) = L_x(Pal(v))x$ and $Pal(wx)=\mu_w(x)Pal(w)$.

For letters $(x_j)_{1 \leq j \leq i}$, formula~\eqref{eq:formula(3)} inductively leads to:

\begin{equation}  \label{eq:u_n2}
Pal(x_1 \cdots x_i) = \mu_{x_1 \cdots x_{i-1}}(x_i) \cdots \mu_{x_1}(x_2) x_1 = \prod_{1 \leq j \leq i}\mu_{x_1 \cdots x_{j-1}}(x_j).
\end{equation}

Note that by convention, $x_1 \cdots x_0=\varepsilon$ in the above product.

\medskip
Now let $\breve w = \breve x_1 \breve x_2 \cdots \breve x_n$ be a spinned version of $w = x_1x_2\cdots x_n$ (viewed as a prefix of a spinned version $\breve \Delta$ of $\Delta$). Then, for any finite word $v$, we have
\begin{equation} \label{eq:f-shift}
  \mu_{\breve w}(v) = S_{\breve w}^{-1}\mu_w(v)S_{\breve w}  \quad \mbox{where $S_{\breve w} = \underset{\underset{\mid \breve x_i=\bar x_i}{i=n, \ldots, 1}}{\prod} \mu_{x_1 \cdots x_{i-1}}(x_i)$.}
\end{equation}

The word $S_{\breve w}$ is called the {\em shifting factor} of $\mu_{\breve w}$ \cite{jJgP04epis}.
Observe that $S_{\breve w}$ is a prefix of $Pal(w)$; in particular $S_{\bar w} = Pal(w)$ by equation~\eqref{eq:u_n2}. Note also that $\mu_{\breve w}(v)= \TT^{|S_{\breve w}|}(\mu_w(v))$.

\medskip

For example, for $\breve w=a \bar b c \bar a$, we have $S_{\breve w}=\mu_{abc}(a) \mu_a(b)=abacabaab$. Thus since $\mu_{abca}(ca) = abacabaab.acabacaba$, $\mu_{a\bar b c \bar a}(ca) = \TT^9(\mu_{abca}(ca)) = acabacaba.abacabaab$.   

\medskip
Likewise, for any infinite word $\by \in \cAw$,
\begin{equation} \label{eq:f-shift-infinite}  
\mu_{\breve w}(\by) = S_{\breve w}^{-1}\mu_w(\by).
\end{equation}

\medskip
This formula used with $\breve w = \bar w$ shows that:

\begin{fact}\label{F:from_referee}
Any word of the form $\mu_w(\by)$ with $\by$ infinite begins with $Pal(w)$.
\end{fact}

\medskip
Amongst the numerous interests of the palindromes $(Pal(x_1 \cdots x_n))_{n \geq 0}$, we have the following explicit characterization of the returns to any factor of an epistandard word.

\begin{thrm} \label{T:return-words} {\em \cite{jJlV00retu}}
Suppose $\bs$ is an epistandard word directed by $\Delta = x_1x_2x_3\cdots$ with $x_i \in \cA$, and consider any factor $v$ of $\bs$. If the word $u_{n+1}=Pal(x_1 \cdots x_n)$ is the shortest palindromic prefix of $\bs$ containing $v$ with $u_{n+1} = fvg$, then the returns to $v$ are given by  $f^{-1}\mu_{x_1\cdots x_n}(x)f$ where $x \in \Alphit(x_{n+1}x_{n+2}\cdots)$. 

\end{thrm}

Because of the uniform recurrence of episturmian words, the following simple but important fact about return words holds.

\begin{lmm} \label{L:important-fact}
Suppose $\bs$ is an epistandard word and let $\bt$ be any episturmian word in the subshift of $\bs$. Then, for any factor $v$ of $\bs$, $r$ is a return to $v$ in $\bs$ if and only if $r$ is a return to $v$ in~$\bt$. 
\end{lmm}

That is, the returns to any factor $v$ of an epistandard word $\bs$ are the same as the returns to $v$ as a factor of any episturmian word $\bt$ with the same set of factors as $\bs$.  
Hereafter, we often use the above result without reference to it.

The following result is particularly useful in Sections~\ref{S:returnWordsEpisturmian}--\ref{SS:quasiperiods}.

\begin{prpstn} \label{P:pali-prefix}  
Suppose $\bs$ is an epistandard word directed by $\Delta = x_1x_2x_3\cdots$ with $x_i \in \cA$, and let $\bt$ be an episturmian word directed by a spinned version of $\Delta$. Then $\bt$ begins with $Pal(x_1 \cdots x_n)$ for some non-negative integer $n$ if and only if $\bt$ has a directive word of the form $\breve\Delta = x_1\cdots x_n \breve x_{n+1} \breve x_{n+2} \cdots$ where the prefix $x_1\cdots x_n$ is $L$-spinned.
\end{prpstn}

\begin{proof}
When $\bt$ is directed by a spinned version of $\Delta$ of the form $\breve\Delta = x_1x_2\cdots x_n \breve x_{n+1} \breve x_{n+2} \cdots$ where the prefix $x_1x_2\cdots x_n$ is $L$-spinned, $\bt$ begins with $Pal(x_1 \cdots x_n)$ by Fact~\ref{F:from_referee}.


Conversely, suppose $Pal(x_1\cdots x_n)$ is a prefix of $\bt$, and suppose that $\bt$ is directed by $\breve \Delta  = (\breve{x}_i)_{i \geq 1}$ (a spinned version of $\Delta = ({x}_i)_{i \geq 1}$). From Theorem~\ref{T:directSame} we can suppose that $\breve \Delta$ contains infinitely many $L$-spinned letters. If $n = 0$, there is nothing to prove. Else let $x_1=a$. Assume first that $\breve{x}_1 = \bar{a}$. 
Let $k$ be the smallest positive integer such that $\breve {x}_k \in {\cA}$. Since $\bt$ begins with the letter $a$ (which is the first letter of $Pal(x_1\cdots x_n)$), we have ${x}_k = {a}$. Then $\breve \Delta = \bar a\bar{x}_2\cdots \bar{x}_{k-1} a\breve x_{k+1}\cdots$, and hence by Theorem~\ref{T:presentation}  $\bt$ is also directed by the spinned infinite word beginning with $a{x}_2\cdots {x}_{j-1}\bar{a}$, with $j \leq k$. Therefore we may assume from now on that $\breve{x}_1 = a$. Let $\bt'$ be the episturmian word directed by $\TT(\breve\Delta) = (\breve{x}_i)_{i \geq 2}$. It is easily seen from the equality $Pal(x_1\cdots x_n) = L_a(Pal(x_2\cdots x_n))a$ that $Pal(x_2\cdots x_n)$ is a prefix of $\bt'$. Hence by induction $\bt'$ is directed by a spinned version of $\TT(\Delta)$ with $x_2\cdots x_n$ as a prefix. And so $\bt$ is directed by a spinned version of $\Delta$ of the form $x_1x_2\cdots x_n\breve x_n\breve x_{n+1}\cdots$ where the prefix $x_1x_2\cdots x_n$ is $L$-spinned.
\end{proof}

\subsection{All epistandard words are quasiperiodic} \label{SS:epistandard-quasiperiods}

A first consequence of Theorem~\ref{T:return-words} is that any epistandard word is quasiperiodic. More precisely:

\begin{thrm} \label{T:epistandard-quasiperiods} Suppose $\bs$ is an epistandard word with directive word $\Delta = x_1x_2x_3\cdots$ with $x_i \in \cA$, and let $m$ be the smallest positive integer such that $\Alphit(x_1x_2\cdots x_{m}) = \Alphit(\bs)$. Then, for all $n \geq m$, $Pal(x_1 \cdots x_n)$ is a quasiperiod of $\bs$.
\end{thrm}

\begin{proof}[Proof of Theorem~$\ref{T:epistandard-quasiperiods}$]  We suppose that $\bs$ is an epistandard word with directive word $\Delta = x_1x_2x_3\cdots$, $x_i \in \cA$. Let $m$ be the smallest positive integer such that $\Alph(x_1x_2\cdots x_{m}) = \Alph(\bs)$.
Clearly, for $n < m$, $Pal(x_1 \cdots x_n)$ cannot be a quasiperiod of $\bs$ since it does not contain all of the letters in $\Alph(\bs)$.

Now let $n \geq m$. We know that $Pal(x_1 \cdots x_n)$ is a prefix of $s$. Suppose that for some $k \geq n$, $Pal(x_1 \cdots x_k)$ is covered by $Pal(x_1 \cdots x_n)$.  Since by choice of $m$, $x_{k+1}$ belongs to $\{x_1, \cdots, x_k\}$, we have $|(Pal(x_1 \cdots x_k)x_{k+1})^{(+)}| \leq 2|Pal(x_1 \cdots x_k)|$. Hence since $Pal(x_1 \cdots x_k)$ is a palindrome, $Pal(x_1 \cdots x_{k+1})$ is covered by $Pal(x_1 \cdots x_k)$, and so by $Pal(x_1 \cdots x_n)$. The result follows from Theorem~\ref{T:epistandard} by induction.

\end{proof}

\begin{xmpl} \label{ex:Trib} Recall the Tribonacci word:
\[
\br = abacabaabacababacabaabacabacabaabaca \cdots~,
\]
which is the epistandard word directed by $(abc)^\omega$. 
Observe that $Pal(abc) = abacaba$ is the shortest palindromic prefix of $\br$ such that $\Alph(Pal(abc)) = \Alph(\br) = \{a,b,c\}$. By Theorem~\ref{T:return-words}, the returns to $Pal(abc)$ in $\br$ are: $\mu_{abc}(a) = abacaba$, $\mu_{abc}(b) = abacab$, $\mu_{abc}(c) = abac$, none of which are longer than $Pal(abc)$. Hence $\br$ is $abacaba$-quasiperiodic; in fact $Pal(abc) = \mu_{abc}(a) = abacaba$ is the smallest quasiperiod of $\br$ since its prefixes $abac$, $abaca$, $abacab$ have returns longer than themselves. This latter fact is also evident from our description of quasiperiods of a (quasiperiodic) episturmian word (Theorem~\ref{T:episturmianQuasiperiods}, to follow).

More generally, the {\em $k$-bonacci word}, which is directed by $(a_1a_2\cdots a_k)^\omega$, is quasiperiodic with smallest quasiperiod $Pal(x_1 \cdots x_k)$. This fact was also observed in \cite{fLgR04quas} by noting that the $k$-bonacci word is generated by the morphism $\varphi_k$ on $\{a_1,a_2,\ldots, a_k\}$ defined by $\varphi_k(a_i) = a_1a_{i+1}$ for all $i \ne k$, and $\varphi_k(a_k) = a_1$.
\end{xmpl}

\begin{rmrk} \label{R:ultimately-L} From Fact~\ref{rnew4} and Theorem~\ref{T:epistandard-quasiperiods}, we immediately deduce that $\varphi(\bs)$ is a quasiperiodic infinite word for any epistandard word $\bs$. Moreover, if $\varphi$ is a pure episturmian morphism, then $\varphi(\bs)$ is a  {\em quasiperiodic episturmian word}. More precisely, $\mu_{\breve w}(\bs)$ is a quasiperiodic episturmian word for any epistandard word $\bs$ and spinned word $\breve w$. Such an episturmian word is directed by a spinned infinite word of the form $\breve w \Delta$ where $\Delta$ is the $L$-spinned directive word of $\bs$. Hence, if an episturmian word $\bt$ is directed by a spinned infinite word with all spins ultimately $L$, then $\bt$ is quasiperiodic.
\end{rmrk}

More generally, we have the following consequence of Theorem~\ref{T:epistandard-quasiperiods} (a converse of this result is stated in Theorem~\ref{T:episturmianQuasiperiods}, and a generalization is provided by Theorem~\ref{T:stronglyQuasEpistandardMorphisms} later).

\begin{crllr} \label{Cor:wvy}
 If an episturmian word $\bt$ is directed by $\breve\Delta = \breve{w}v\breve \by$ for some spinned words $\breve w$, $\breve \by$ and $L$-spinned word $v$ such that $\Alphit(v) = \Alphit(v\by)$, then $\bt$ is quasiperiodic. 
 
\noindent
Moreover, any word of the form $\mu_{\breve w}(Pal(v))p$ with $p$ a prefix of $S^{-1}_{\breve w}Pal(w)$ is a quasiperiod of $\bt$.
\end{crllr}
\begin{proof} Let $(\bt^{(n)})_{n\geq 0}$ be the infinite sequence of episturmian words associated to $\bt$ and $\breve \Delta$ in Theorem~\ref{T:episturmian}. Then by Proposition~\ref{P:pali-prefix}, the episturmian word $\bt^{(|w|)}$, which is directed by $v\breve \by$,  begins with the palindromic prefix $Pal(v)$ of the epistandard word $\bs^{(|w|)}$ directed by the $L$-spinned version of $v\breve \by$. Moreover, since $\Alph(v) = \Alph(v\by)$, $Pal(v)$  is a quasiperiod of  $\bs^{(|w|)}$ by Theorem~\ref{T:epistandard-quasiperiods}. By Fact~\ref{F:episturmian}, $\bs^{(|w|)}$ is in the subshift of $\bt^{(|w|)}$. It follows from Lemma~\ref{L:important-fact} that $Pal(v)$ has the same returns in $\bt^{(|w|)}$ as it does in $\bs^{(|w|)}$, and since it is a prefix of $\bt^{(|w|)}$, it is also a quasiperiod of $\bt^{(|w|)}$.  
Therefore by Fact~\ref{rnew4} $\bt = \mu_{\breve w}(\bt^{(|w|)})$ is quasiperiodic and $\mu_{\breve w}(Pal(v))$ is a quasiperiod of $\bt$.


Now by formula~\eqref{eq:f-shift-infinite} and by Fact~\ref{F:from_referee}, each occurrence of $\mu_{\breve w}(Pal(v))$ is followed by $S^{-1}_{\breve w}Pal(w)$ in $\bt$, and so is followed by $p$ for any prefix $p$ of $S^{-1}_{\breve w}Pal(w)$. This shows that $\mu_{\breve w}(Pal(v))p$ is a quasiperiod of~$\bt$.
\end{proof}

\subsection{Ultimate quasiperiods}

Theorem \ref{T:epistandard-quasiperiods} shows that long enough palindromic prefixes of an epistandard word $\bs$ are quasiperiods of~$\bs$. Our next goal is to extend Theorem~\ref{T:epistandard-quasiperiods} by describing all the quasiperiods of any (quasiperiodic) episturmian word. Several lemmas are required, using the notion of ultimate quasiperiods that we now define.

A factor $v$ of an infinite word $\bw$ is said to be an {\it ultimate quasiperiod} of $\bw$ if it is a quasiperiod of a suffix of $\bw$.
In particular, when $\bw$ is a uniformly recurrent word (which is the case for episturmian words), a recurrent factor $v$ in $\bw$ is an ultimate quasiperiod of $\bw$ if any return to $v$ in $\bw$ has length at most $|v|$.

Clearly, if an ultimate quasiperiod is a prefix of $\bw$, then $\bw$ is quasiperiodic by Lemma~\ref{L:quasi}. Also note that the set of quasiperiods of a (quasiperiodic) infinite word $\bw$ consists of all its ultimate quasiperiods that are prefixes of it.

\medskip
Remark that we have
\begin{prpstn}
All episturmian words are ultimately quasiperiodic.
\end{prpstn}

\begin{proof}
From Theorem~\ref{T:epistandard-quasiperiods}, all epistandard words are quasiperiodic. Let $\bs$ be an epistandard word and let $q$ be a quasiperiod of $\bs$. Then any episturmian word $\bt$ in the subshift of $\bs$ has a suffix $\bt'$ beginning with $q$. Since $\bt$ has the same set of factors than $\bs$, we can find infinitely many prefixes of $\bt'$ which are covered by $q$, i.e., $\bt'$ is quasiperiodic. Hence $\bt$ is ultimately quasiperiodic.
\end{proof}

Now we recall some insight about palindromic closure that will be useful later.
As previously, let $\bs$ denote an epistandard word with directive word $\Delta = x_1x_2x_3\cdots$ with $x_i \in \cA$. For $n \geq 0$ let $u_{n+1}=Pal(x_1 \cdots x_n)$. Note in particular that $u_1=\varepsilon$ and by Theorem~\ref{T:epistandard}, $\bs = \lim_{n\rightarrow \infty} u_n$.

As in \cite{jJlV00retu, jJgP02epis}, let us define $P(i) = \sup\{j < i \mid x_j = x_i\}$ if this number exists, undefined otherwise. That is, if $x_i = a$, then $P(i)$ is the position of the right-most occurrence of the letter $a$ in the prefix $x_1x_2\cdots x_{i-1}$ of the directive word $\Delta$. For instance, if $\Delta = (abc)^\omega$, $P(i) = i-3$ for any $i \geq 4$ and $P(1)$, $P(2)$, and $P(3)$ are undefined.

From the definitions of palindromic closure and the palindromes $(u_{i})_{i\geq 1}$, it follows that, for all $i \geq 1$, 
\begin{equation} \label{eq:u_{i+1}}
u_{i+1} = \begin{cases}
                     u_ix_iu_i &\textrm{if $x_i \notin \Alph(u_i)$,} \\
                     u_iu_{P(i)}^{-1}u_i &\textrm{otherwise}.
                \end{cases}
\end{equation}

Therefore, using Theorem~\ref{T:return-words} with $f=g=\varepsilon$, we deduce that for $n\geq 0$, the length of the longest return $r_{n+1}$ to $u_{n+1}$ in $\bs$ satisfies
\[
  |r_{n+1}| = \begin{cases} 
                |u_{n+1}| + 1 &\mbox{if some $x \in \Alph(\bs)$ does not occur in $u_{n+1}$}, \\
                |u_{n+1}| - |u_{p_n}| &\mbox{otherwise},
                \end{cases}
\]
where $p_n = \inf\{P(i) \mid i\geq n+1\}$ (see also \cite[Lem.~5.6]{jJlV00retu}).

\medskip
In other words, $p_n = \sup\{i\leq n  \mid \Alph(x_i\cdots x_{n}) = \Alph(x_i\cdots x_{n} \cdots)\}$. For instance, if $\Delta = (abc)^\omega$, then $p_4 = 1$.

\medskip
The next lemma 
gives the set of all ultimate quasiperiods of any episturmian word $\bt$. It simply amounts to determining all of the factors of $\bt$ that have no returns longer than themselves. 

\begin{notation} Hereafter, we denote by $F(w)$ the set of factors of a finite or infinite word $w$.
\end{notation}

\begin{lmm} \label{L:ultimate_quasiperiods} Suppose $\bs$ is an epistandard word directed by $\Delta = x_1x_2x_3\cdots$, $x_i \in \cA$. Let $m$ be the smallest positive integer such that $\Alphit(x_1x_2\cdots x_m) = \Alphit(\bs)$ and let $u_{n+1}=Pal(x_1 \cdots x_n)$ for all $n \geq 0$. Then the set of all ultimate quasiperiods of any episturmian word $\bt$ in the subshift of $\bs$ is given by 
\[
  \cQ = \bigcup_{n\geq m} \cQ_n \quad \mbox{with}\quad  \cQ_n = \{ q \in F(u_{n+1}) \mid |q| \geq |u_{n+1}| - |u_{p_n}|\},   
\] 
where $p_n = \sup\{i\leq n \mid \Alphit(x_i\cdots x_{n}) = \Alphit(x_i\cdots x_{n} \cdots)\}$.
\end{lmm}
\begin{proof} First observe that the number $p_n$ exists for all $n\geq m$. Indeed, the set $\{i\leq n \mid \Alph(x_i\cdots x_{n}) = \Alph(x_i \cdots x_n \cdots)\}$ is not empty, as by the definition of $m$ it contains $i=1$.

Clearly, if $n < m$, then no factor of $u_{n+1}$ can be an ultimate quasiperiod of $\bt$  since $u_{n+1}$ does not contain all of the letters in $\Alph(\bt) = \Alph(\Delta)$. So let us now fix $n \geq m$. For any $q\in \cQ_n$, $|q| \geq |u_{n+1}| - |u_{p_n}| > |u_{n}|$. Indeed from formula~\eqref{eq:u_{i+1}}, if $n = m$ then $u_{n+1}= u_nx_nu_n$, and if $n \geq m+1$, then $u_{n+1} = u_{n}u_{P(n)}^{-1}u_{n}$ where $P(n) \geq {p_n}$. Hence $q \in F(u_{n+1})\setminus F(u_n)$, i.e., $u_{n+1}$ is the shortest palindromic prefix of $\bs$ containing $q$. Therefore, by Theorem~\ref{T:return-words}, the returns to $q \in \cQ_n$ are a certain circular shift of the returns to $u_{n+1}$ and the longest of these return words has length $|u_{n+1}| - |u_{{p_n}}|$. Thus any return to $q$ has length at most $|q|$; 
whence $q$ is an ultimate quasiperiod of $\bt$. It remains to show that any other factor $w \in F(u_{n+1})\setminus F(u_n)$ with $|w| < |u_{n+1}| - |u_{{p_n}}|$ is not an ultimate quasiperiod. This is clearly true since the longest return to any such $w$ has length $|u_{n+1}| - |u_{{p_n}}| > |w|$. That is, at least one of the returns to $w$ is longer than it,  which implies that $w$ is not an ultimate quasiperiod of $\bt$.
\end{proof}

\begin{xmpl} Let us consider the Fibonacci case. As observed in \cite{aD81acom}, $|u_n| = F_{n+1}-2$ for all $n \geq 1$ where $F_k$ is the $k$-th Fibonacci number ($F_1 = 1$, $F_2 = 2$, $F_k = F_{k-1}+F_{k-2}$ for $k \geq 2$). Since for $n \geq 2$, 
$p_n = n-1$ and  $|u_{n+1}| - |u_{p_n}| = F_{n+1}$, the
ultimate quasiperiods of the Fibonacci word are the factors of $u_{k}$ of
length between $F_k$ and $F_{k+1} -2$ for all $k \geq 3$. The first
few ultimate quasiperiods of the Fibonacci word (in order of increasing length) are: $aba$ ($u_3$), $abaab$, $baaba$, $abaaba$ ($u_4$), $abaababa$, $baababaa$, $aababaab$, $ababaaba$, $abaababaa$, $baababaab$, $aababaaba$, $abaababaab$, $baababaaba$, $abaababaaba$ ($u_5$), $\ldots$
\end{xmpl}

Lemma~\ref{L:ultimate_quasiperiods} yields the following trivial characterization of quasiperiodic episturmian words.

\begin{crllr} \label{Cor:quasi-char} Suppose $\bs$ is an epistandard word with set of ultimate quasiperiods $\cQ$. Then an episturmian word $\bt$ in the subshift of $\bs$ is quasiperiodic if and only if some $v \in \cQ$ is a prefix of~$\bt$. 
\end{crllr}

Moreover, Lemma \ref{L:ultimate_quasiperiods} can be reformulated (more nicely) using episturmian morphisms, together with the iterated palindromic closure function.

\begin{lmm} \label{L:ultimate_quasiperiods-2} Suppose $\bs$ is an epistandard word directed by $\Delta \in \cAw$. Then the set of ultimate quasiperiods of any episturmian word $\bt$  in the subshift of $\bs$  is the set of all  words 
\[
  q \in F(Pal(wv)), \quad \mbox{with $|q| \geq |\mu_w(Pal(v))|$},
\]  
where $w$, $v$ are words such that $\Delta = wv\by$ with $\Alphit(v) = \Alphit(v\by)$.
\end{lmm}
\begin{proof}  Let $\Delta = x_1x_2x_3\cdots$ and let $m$ be the smallest positive integer such that $\Alph(x_1x_2\cdots x_{m}) = \Alph(\bs)$. Then, by Lemma~\ref{L:ultimate_quasiperiods}, the set of all ultimate quasiperiods of $\bs$ (and hence of $\bt$) is given by $\cQ = \bigcup_{n\geq m}\cQ_n$ with
\[
  \cQ_n = \{ q \in F(u_{n+1}) \mid |q| \geq |u_{n+1}| - |u_{p_n}|\}, 
\] 
where  $p_n = \sup\{i\leq n  \mid \Alph(x_i\cdots x_{n}) = \Alph(x_i\cdots x_{n} \cdots)\}$. So, for fixed $n \geq m$, $\Delta$ can be written as $\Delta = wv\by$ where $w = x_1\cdots x_{{p_n}-1}$, $v = x_{{p_n}}\cdots x_n$, $\by = x_{n+1}x_{n+2}\cdots$, and $\Alph(v) = \Alph(v\by)$ (by the definition of $p_n$). Then $u_{n+1} = Pal(x_1 \cdots x_n) = Pal(wv)$ and $u_{p_n} = Pal(x_1 \cdots x_{p_n}) = Pal(w)$.
So now, using formula \eqref{eq:formula(3)}, we have $Pal(wv)= \mu_w(Pal(v))Pal(w)$;
in particular, $|u_{n+1}| - |u_{p_n}| = |Pal(wv)| - |Pal(w)| = |\mu_{w}(Pal(v))|$. 
Thus, Lemma~\ref{L:ultimate_quasiperiods} tells us that for $q \in \cQ_n$, there exist words $w$, $v$ such that $q \in 
 F(Pal(wv))$ and  $|q| \geq |\mu_w(Pal(v))|$.

Conversely, assume $\Delta = wv\by$ with $\Alph(v) = \Alph(v\by)$. Let $v_1$, $v_2$ be such that $v = v_1v_2$ with $v_2$ the smallest suffix of $v$ such that $\Alph(v_2) = \Alph(v_2\by)$. Let $n = |wv|$. It follows from the choice of $v_2$ that $p_n = |wv_1|$+1. By Theorem~\ref{T:epistandard-quasiperiods}, $Pal(v_2)$ is a quasiperiod of the epistandard word directed by $v_2\by$. By Fact~\ref{rnew4} $\mu_{wv_1}(Pal(v_2))$ is a quasiperiod of $\bs$. Since any occurrence of $\mu_{wv_1}(Pal(v_2))$ in $\bs$ is followed by $Pal(wv_1)$, we deduce that any factor of $\mu_{wv_1}(Pal(v_2))Pal(wv_1)$ of length greater than $|\mu_{wv_1}(Pal(v_2))|$ is an ultimate quasiperiod of $\bt$. Using formula~\eqref{eq:formula(3)}, we see that $Pal(wv_1) = \mu_{w}(Pal(v_1))Pal(w)$ so that
$\mu_{wv_1}(Pal(v_2)) = \mu_{w}(\mu_{v_1}(Pal(v_2))Pal(v_1)) = \mu_{w}(Pal(v_1v_2)) = \mu_w(Pal(v))$. Hence any factor of $\mu_{w}(Pal(v))Pal(w) = Pal(wv)$ of length greater than $|\mu_w(Pal(v))|$ ($\geq |\mu_{wv_1}(Pal(v_2))|$) is an ultimate quasiperiod of $\bt$.
\end{proof}

\subsection{Quasiperiods of episturmian words} \label{SS:quasiperiods}

We are now ready to state the main theorem of this section, which describes all of the quasiperiods of an episturmian word.

\begin{thrm}  \label{T:episturmianQuasiperiods} 
The set of quasiperiods of an episturmian word $\bt$ is the set of all words 
\begin{equation}\label{eq:words}
  \mu_{\breve w}(Pal(v))p, \quad \mbox{with $p$ a prefix of $S_{\breve w}^{-1}Pal(w)$},
\end{equation}
where $w$, $v$ are $L$-spinned words such that $\bt$ is directed by $\breve wv\breve \by$ for some spinned version $\breve w$ of $w$ and some spinned version $\breve\by$ of an $L$-spinned infinite  word $\by$ with $\Alphit(v) = \Alphit(v\by)$. 

Moreover, the smallest quasiperiod of $\bt$ is the word $\mu_{\breve w}(Pal(v))$ where $wv$ is of minimal length for the property $\Alphit(wv) = \Alphit(wv\by)$, and amongst all decompositions of $wv$ into $w$ and $v$, the word $v$ is the shortest suffix of $wv$ such that $\Alphit(v) = \Alphit(v\by)$.
\end{thrm}

The above theorem shows that if there do not exist words $\breve w$, $\by$, $\breve \by$, and $v$ (as defined above) such that $\bt$ is directed by $\breve w v \breve \by$ with $\Alph(v) = \Alph(v\by)$, then $\bt$ does not have any quasiperiods, and hence $\bt$ is non-quasiperiodic. For instance, any regular wavy word $\breve\Delta$ (recall the definition from Section~\ref{SS:uniqueDirective}) clearly directs an episturmian word with no quasiperiods since $\breve\Delta$ is the only directive word for $\bt$ and it does not contain an $L$-spinned factor $v$ containing all letters that follow it in $\breve\Delta$. For example, $(a\bar b c)^\omega$ directs a non-quasiperiodic episturmian word in the subshift of the Tribonacci word.

Let us now illustrate the last part of the theorem. For the epistandard word directed by $\Delta = ca(ab)^\omega$, i.e., the image of the Fibonacci word by the morphism $L_cL_a$, the shortest word $wv$ such that $\Delta = wv\by$ for some infinite word $\by$ with $\Alph(wv) = \Alph(wv\by)$ is the word $caab$. There are three ways to decompose this word $caab$ into $wv$: 1)~$w = \varepsilon$ and $v = caab$;   
2)~$w = c$ and $v = aab$; 3)~$w = ca$ and $v = ab$. The corresponding quasiperiods of the form $\mu_w(Pal(v))$ are respectively: 1)~$Pal(caab) = cacacbcacac$; 2)~$\mu_{c}(Pal(aab)) = cacacbcaca$; 3)~$\mu_{ca}(Pal(ab)) = cacacbca$.

\begin{proof}[Proof of Theorem~$\ref{T:episturmianQuasiperiods}$]
Corollary~\ref{Cor:wvy} proves that any word of the form~\eqref{eq:words} is a quasiperiod of $\bt$.

Now suppose that $q$ is a quasiperiod of $\bt$. We show that $\bt$ has at least one directive word of the form $\breve w v\breve \by$ where $\breve w$, $\breve \by$ are spinned versions of some $L$-spinned words $w$, $\by$, and where $v$ is an $L$-spinned word such that $\Alph(v) = \Alph(v\by)$. Moreover, we show that $q = \mu_{\breve w}(Pal(v))p$ for some prefix $p$ of $S_{\breve w}^{-1}Pal(w)$.  

Let $\bs$ be an epistandard word with $L$-spinned directive word $\Delta$ such that $\bt$ is directed by a spinned version of $\Delta$. 
The quasiperiod $q$ is an ultimate quasiperiod of $\bt$ that occurs as a prefix of $\bt$. So, by Lemma~\ref{L:ultimate_quasiperiods-2}, 
\begin{equation*} 
q \in F(Pal(wv)), \quad \mbox{with} \quad |q| \geq |\mu_{w}(Pal(v))|,
\end{equation*}
for some $L$-spinned words $w$, $v$, and $\by$  such that $\Delta = wv\by$ with $\Alph(v) = \Alph(v\by)$. In particular, we have $Pal(wv) = fqg$ for some words $f$, $g$ with $|f| + |g| \leq |Pal(w)|$, where $Pal(wv) = \mu_w(Pal(v))Pal(w)$.
By definition of the $Pal$ function, $Pal(w)$ is a prefix of $Pal(wv)$. Consequently $f$ is a prefix of $Pal(w)$.

Note that
\begin{equation} \label{eq:quasi-2}
  q = f^{-1}\mu_w(Pal(v))Pal(w)g^{-1}.
\end{equation}  
Thus $q = (f^{-1}\mu_{w}(Pal(v))f)p$ with $p := f^{-1}Pal(w)g^{-1}$ a prefix of $f^{-1}Pal(w)$.

\medskip

The following result 
shows that $f = S_{\breve w}$ for some spinned version $\breve w$ of $w$.

\begin{lmm}
Given a word $w$ and a prefix $f$ of $Pal(w)$, 
there exists a spinned version $\breve w$ of $w$ such that $f = S_{\breve w}$.
\end{lmm}
\begin{proof} The proof proceeds by induction on $|w|$. The lemma is clearly true for $|w| = 0$ since in this case $f$ is a prefix of $Pal(w) = Pal(\empt) = \empt$, and hence $f = \empt =  S_{\empt}$. 

Now suppose $|w| \geq 1$ and let us write $w = xw'$ where $x$ is a letter. Since $f$ is a prefix of $Pal(w) = Pal(xw') = \mu_x(Pal(w'))x$ (see formula~\eqref{eq:formula(3)}), we have $f = \mu_x(f')$ or $f = \mu_{x}(f')x$ for some prefix $f'$ of $Pal(w')$. Moreover, by the induction hypothesis, $f' = S_{\breve w'}$ for some spinned version $\breve w'$ of $w'$. Hence, using formula~\eqref{eq:f-shift}, we have $f = \mu_x(S_{\breve w'}) = S_{x\breve w'} \quad \mbox{or} \quad f = \mu_x(S_{\breve w'})x = S_{\bar x\breve w'}$.
That is, $f = S_{\breve w}$ for some spinned version $\breve w$ of $w = xw'$.
\end{proof}

Now we have to prove that $\bt$ is directed by a spinned infinite word beginning with $\breve wv$.
For this we need some further intermediate results, as follows.

\begin{lmm} \label{L:u_w}
For any word $u$ containing at least two different letters and for any other word $w$, there exists a word $u_w$ containing at least two different letters such that $\mu_w(u) = Pal(w)u_w$.
\end{lmm}
\begin{proof} The proof proceeds by induction on $|w|$. The lemma is trivially true for $|w| = 0$. Now suppose $|w| \geq 1$ and let us write $w = xw'$ where $x$ is a letter. Then $\mu_w(u) = \mu_{xw'}(u) = \mu_x(\mu_{w'}(u))$ where, by the induction hypothesis, $\mu_{w'}(u) = Pal(w')u_{w'}$ for some word $u_{w'}$ containing at least two different letters. Hence, $\mu_{w}(u) = \mu_x(Pal(w')u_{w'}) = \mu_x(Pal(w'))\mu_x(u_{w'})$, and therefore by formula~\eqref{eq:formula(3)} we have $\mu_w(u) = Pal(xw')x^{-1}\mu_x(u_{w'})$ where the word $x^{-1}\mu_x(u_{w'})$ contains at least two different letters. 
This completes the proof of the lemma.
\end{proof}

\begin{crllr}
\label{C:smallestPalindrom}
For any letter $x$ and for any words $v$, $w$ such that $v$ contains at least one letter different from $x$, 
\begin{equation} \label{eq:smallestPalindrom}
|\mu_{w}(Pal(vx))| \geq |Pal(wv)| + 2.
\end{equation}
\end{crllr}

\begin{proof}
We distinguish two cases: $x \not\in \Alph(v)$, $x \in \Alph(v)$. If $x \not\in \Alph(v)$, then $Pal(vx) = Pal(v) xPal(v)$ and the word $u = xPal(v)$ contains at least two different letters. On the other hand, if $x \in \Alph(v)$, then $v = v_1xv_2$ where $v_2$ is $x$-free, in which case $Pal(vx) = Pal(v) Pal(v_1)^{-1}Pal(v)$ and the word $u = Pal(v_1)^{-1}Pal(v)$ contains at least two different letters. Hence, in either case, $Pal(vx) = Pal(v) u$ where $u$ is a word containing at least two different letters. Thus, it follows from Lemma~\ref{L:u_w} that $\mu_w(Pal(vx)) = \mu_w(Pal(v))Pal(w)u_w$ for some word $u_w$ containing at least two different letters. Since $Pal(wv) = \mu_w(Pal(v))Pal(w)$ and $|u_w| \geq 2$, the proof is thus complete.
\end{proof}

\begin{note} Inequality~\eqref{eq:smallestPalindrom} is not true in general. For instance, when $v = \varepsilon$ and $w = xx$, we have $\mu_w(Pal(vx)) = x$ and $Pal(wv) = xx$.
\end{note}

\medskip

Let us come back to our proof of Theorem~\ref{T:episturmianQuasiperiods}. Writing $v = v'x$, we have $|q| \geq |\mu_{w}(Pal(v))| > |Pal(wv')|$. Hence as an immediate consequence of Corollary~\ref{C:smallestPalindrom}, when $v$ contains at least two different letters, 
$Pal(wv)$ is the smallest palindromic prefix of $\bs$ of which $q$ is a factor. Therefore by Theorem~\ref{T:return-words} and Lemma~\ref{L:important-fact}, the returns to $q$ in $\bt$ are the words $f^{-1}\mu_{wv}(\alpha)f$ where $\alpha \in \Alph(\by)$. Consequently, each occurrence of $q$ is preceded by the word $f$. Thus, the set of factors of $f\bt$ is exactly the same as the set of factors of $\bt$; whence, the infinite word $f\bt$ (which is clearly recurrent) is episturmian. Moreover, returns to $fq$ in $f\bt$ are of the form $\mu_{wv}(\alpha)f$ for letters $\alpha \in \Alph(\by)$. Hence we deduce that there exists an infinite word $\bt'$ such that $f\bt = \mu_{wv}(\bt')$. Moreover, we deduce from the following lemma that $\bt'$ is episturmian.

\begin{lmm} \label{L:pureEpisturmian}
For any letter $\alpha$, an infinite word $\bw$ is episturmian if and only if $\mu_\alpha(\bw)$ is episturmian.
\end{lmm}
\begin{proof}
$(\Rightarrow)$: Immediately follows from Theorem~\ref{T:episturmian} (see also Corollary~3.12 in \cite{jJgP02epis} which shows more generally that if $\bw$ is episturmian and $\varphi$ is an episturmian morphism, then $\varphi(\bw)$ is episturmian). \medskip

\noindent $(\Leftarrow)$: Conversely,  suppose $\bz := \mu_\alpha(\bw)$ is an episturmian word. Then $\bz$ begins with the letter $\alpha$. Hence, by Proposition~\ref{P:pali-prefix},  $\bz$ is directed by a spinned infinite word $\breve\Delta$ beginning with $\alpha$, say $\breve\Delta = \alpha\breve\by$ for some spinned infinite word $\breve\by$. So by Theorem~\ref{T:episturmian} and Remark~\ref{R:directive}, $\bz = \mu_\alpha(\bz')$ where $\bz'$ is an episturmian word directed by $\breve\by$. By the injectivity of $\mu_\alpha$, $\bz' = \bw$; whence $\bw$ is episturmian.
\end{proof}

Now, since $f = S_{\breve w}$, we have $\bt = \mu_{\breve wv}(\bt')$ by formula~\eqref{eq:f-shift-infinite}, and so $\bt$ is directed by a spinned infinite word beginning with $\breve w v$.

It remains to consider the case when $v$ is a power of a letter $x$. In this case, the condition $\Alph(v) = \Alph(v\by)$ implies that $\by = x^\omega$; thus $\bs$ (and hence $\bt$) is periodic. More precisely,  $\bs = (\mu_{w}(x))^\omega$ and $q = f^{-1}\mu_{w}(x)f$. Since  $f = S_{\breve w}$, it follows from formula~\eqref{eq:f-shift} that $q = \mu_{\breve w}(x)$ and $\bt = q^\omega$ showing that $\bt$ is directed by $\breve wx^\omega$. 

The proof of Theorem~\ref{T:episturmianQuasiperiods} is thus complete, except for the last part concerning the smallest quasiperiod which we prove below.

\medskip

As previously, let $\bs$ be an epistandard word such that $\bt$ belongs to the subshift of $\bs$ and let  $\Delta$ be the $L$-spinned directive word of $\bs$. 
First of all, we observe that the smallest quasiperiod of $\bt$ is of the form
$q := \mu_{\breve w}(Pal(v))$, where $\breve w$ is a spinned version of a word $w$ such that $\Delta = wv\by$ and $\Alph(v) = \Alph(v\by)$. Assume first that $v$ has a proper suffix $s$ such that $\Alph(s) = \Alph(s\by)$, so that writing $v = ps$ for a non-empty word $p$, $\mu_{\breve w p}(Pal(s))$ is also a quasiperiod of $\bt$. By formula~\eqref{eq:formula(3)},
$\mu_{p}(Pal(s))$ is a proper prefix of $Pal(ps) = Pal(v)$. Thus $\mu_{\breve w p}(Pal(s))$ is a proper prefix of  $\mu_{\breve w}(Pal(v))$ showing that the second of these words is not the smallest quasiperiod of $\bt$.

From now on, we consider $L$-spinned words $w_1$, $w_2$, $v_1$, $v_2$, spinned versions $\breve{w}_1$ and $\breve{w}_2$ of $w_1$ and $w_2$ respectively, spinned words $\by_1$ and $\by_2$ such that $\mu_{\breve{w}_1}(Pal(v_1))$ and $\mu_{\breve{w}_2}(Pal(v_2))$ are quasiperiods of $\bt$, $\Delta = w_1v_1\by_1 = w_2v_2\by_2$, $\Alph(v_1) = \Alph(v_1\by_1)$,  and $\Alph(v_2) = \Alph(v_2\by_2)$. Moreover we assume that $v_1$ and $v_2$ verify the following hypothesis:
\begin{description}
\item{(H)} for $i = 1, 2$, $v_i$ has no proper suffix $s_i$ with $\Alph(s_i) = \Alph(s_i\by_i)$
\end{description}

 Note that $|\mu_{\breve{w}_1}(Pal(v_1))| = |\mu_{{w}_1}(Pal(v_1))|$ and $|\mu_{\breve{w}_2}(Pal(v_2))| = |\mu_{{w}_2}(Pal(v_2))|$, so that to determine whether $\mu_{\breve{w}_1}(Pal(v_1))$ or $\mu_{\breve{w}_2}(Pal(v_2))$ is the smallest quasiperiod, we just have to determine which of the two words $\mu_{{w}_1}(Pal(v_1))$ and $\mu_{{w}_2}(Pal(v_2))$ is the shortest word. But then we can use the fact that $w_1v_1$ and $w_2v_2$ are both prefixes of $\Delta$. When $|w_1| = |w_2|$, Hypothesis (H) implies that $w_1 = w_2$ and $v_1 = v_2$.

Before considering the case where $|w_1v_1| < |w_2v_2|$, let us recall that if $\alpha$ is a letter and if $x$ is an $\alpha$-free word, then $\mu_x(\alpha) = Pal(x)\alpha$ (see formula~\eqref{eq:formula(3)}). 
Moreover, we observe that for any word $x$ and distinct letters $\alpha$ and $\beta$, $Pal(x\alpha)$ is a prefix of $\mu_x(Pal(\alpha\beta))$. Indeed, $\mu_x(Pal(\alpha\beta)) = \mu_x(\mu_\alpha(\beta)\alpha) = \mu_{x\alpha}(\beta\alpha)$ and  this word contains $Pal(x\alpha)$ as a prefix by Lemma~\ref{L:u_w}.

Assume $|w_1v_1| < |w_2v_2|$ (the case $|w_1v_1| > |w_2v_2|$ is symmetric) and let us show that $|\mu_{\breve{w}_1}(Pal(v_1))|$ is less than (or equal to) $|\mu_{\breve{w}_2}(Pal(v_2))|$.
In this case $w_1v_1$ is a proper prefix of $w_2v_2$. Let $x$ be the non-empty word such that $w_1v_1x = w_2v_2$. Hypothesis (H) implies that $v_1$ cannot be a factor of $v_2$ except as a prefix. Thus there exists a possibly empty word $y$ such that $v_1x = yv_2$. Assume $x$ has length at least 2 and is a suffix of $v_2$. 
Then Hypothesis (H) implies that the two last letters of $v_2$ and so of $x$ are different. Let $\alpha$ and $\beta$ be the two last letters of $x$, and let $x'$ be the word such that $x = x'\alpha\beta$. 
Then $Pal(v_1)$ is a prefix of $Pal(v_1x')$ which is  a prefix of $\mu_{v_1x'}(\alpha\beta)$ by Lemma~\ref{L:u_w}. Finally we can see that $\mu_{v_1x'}(\alpha\beta)$ is a prefix of $\mu_{y}(Pal(v_2))$. This shows that $Pal(v_1)$ is a proper prefix of $\mu_{y}(Pal(v_2))$ and so $\mu_{\breve{w}_1}(Pal(v_1))$ is a shorter quasiperiod than $\mu_{\breve{w}_2}(Pal(v_2))$. 

Now consider $|x| = 1$ and $|v_2| \geq 2$. Let $\alpha$ be the last letter of $v_1$ and let $v_1'$ and $v_2'$ be the words such that $v_1 = v_1' \alpha$, and $v_2 = v_2'\alpha x$. Then 
$Pal(v_1) = Pal(v_1' \alpha)$ is a proper prefix of
$\mu_{v_1'}(Pal(\alpha x)) = \mu_{yv_2'}(Pal(\alpha x))$ (see the paragraph before last, taking $\alpha\beta = \alpha x$),   which is a prefix of $\mu_{y}(Pal(v_2))$. So again we find that $\mu_{\breve{w}_1}(Pal(v_1))$ is a shorter quasiperiod than $\mu_{\breve{w}_2}(Pal(v_2))$. 

Now we come to the case when $v_2 = \alpha$ for a letter $\alpha$. If $x$ contains a letter different from $\alpha$, then $Pal(v_1)$ is a proper prefix of $\mu_{v_1}(\mu_{x}(\alpha)) = \mu_{y}(\alpha) = \mu_{y}(Pal(v_2))$. We still conclude as previously. 

Lastly, assume that $x$ is a power of $\alpha$. Hypothesis (H) implies that $\alpha$ is the first letter of $v_1$ and more precisely $v_1 = \alpha v_1'$ with $v_1'$ an $\alpha$-free word. Thus  
$Pal(v_1) = \mu_{\alpha}(Pal(v_1')) \alpha
           = \mu_{\alpha}(Pal(v_1') \alpha) 
           = \mu_{\alpha}(\mu_{v_1'}(\alpha)) 
            = \mu_{\alpha v_1'}(\alpha) = \mu_{v_1}(x) = \mu_{y}(v_2)$. So in this particular case 
$\mu_{\breve{w}_1}(Pal(v_1)) = \mu_{\breve{w}_2}(Pal(v_2))$, and once again $\mu_{\breve{w}_1}(Pal(v_1))$ is the smallest quasiperiod.
\end{proof}

\begin{rmrk} \label{R:no-of-quasiperiods}
Theorem \ref{T:episturmianQuasiperiods} shows in particular that if an episturmian word $\bt$ is directed by a spinned word $\breve\Delta$ with all spins ultimately $L$, then $\bt$ has infinitely many quasiperiods since there are infinitely many factorizations of $\breve\Delta$ into the given form (i.e., any {\em positive shift} of an epistandard word is quasiperiodic). 
\end{rmrk}

Let us demonstrate Theorem~\ref{T:episturmianQuasiperiods} with some examples. First we provide an example of an episturmian word having infinitely many quasiperiods, but which is  not epistandard and all of its directive words are {\em wavy} (recall that a wavy word is a spinned infinite word containing  infinitely many $L$-spinned letters and infinitely many $R$-spinned letters).

\begin{xmpl} \label{ex:1}
Consider the episturmian word $\bt$ with normalized directive word $\breve\Delta = abc(\bar{b}a)^\omega$. From Theorem~\ref{T:directSame}, we observe that $\bt$ has infinitely many directive words:
\[
 \breve\Delta = abc(\bar{b}a)^\omega \quad \mbox{and} \quad \breve\Delta_i = a\bar b \bar c (\bar b \bar a)^iba(\bar b a)^\omega \quad \mbox{for each $i\geq0$},
 \]
 all of which are wavy.
Hence, by Theorem~\ref{T:episturmianQuasiperiods}, the set of quasiperiods of $\bt$ consists of $Pal(abc) = abacaba$ and, for each $i\geq 0$, the words:
\[
\mu_{\breve w}(bab) \quad \mbox{and} \quad \mu_{\breve w}(bab)a \quad \mbox{where $\breve w = a\bar b \bar c (\bar b \bar a)^i$}.
\]
Note that $Pal(ba) = bab$ and $S_{\breve w}^{-1}Pal(w) = (Pal(w)a^{-1})^{-1}Pal(w) = a$. 
\end{xmpl}

Next we give some examples of quasiperiodic episturmian words having only finitely many quasiperiods.

\begin{xmpl} \label{ex:epi-quasiperiods} Consider the quasiperiodic episturmian word $\mu_{adbc\bar d}(a\br)$ where $\br$ is the Tribonacci word;  
it has only two directive words: $adbc\bar d(a\bar b\bar c)^\omega$ and $a\bar d \bar b \bar c d(a\bar b \bar c)^\omega$ (the first one being its normalized directive word). So we see that $\bt$ has {\em only} one quasiperiod, namely $Pal(adbc) = adabadacadabada$. Similarly, the episturmian word $\mu_{cbaa}(c\br)$, which is directed by exactly three different spinned infinite words ($cbaa(\bar{a}\bar{b}c)^\omega$, $cba\bar{a} a(\bar{b}c\bar{a})^\omega$, and $cb\bar{a} a a(\bar{b}c\bar{a})^\omega$), has only two quasiperiods: $Pal(cba) = cbcacbc$ and $Pal(cbaa) = cbcacbcacbc$.
\end{xmpl}

We can also construct episturmian words having exactly $k$ quasiperiods for any fixed integer $k \geq 1$, as shown by the following example.

\begin{xmpl} \label{ex:3}
Consider the episturmian word with normalized directive word $\breve \Delta = a b c\bar a(b\bar c \bar a)^\omega$. By Theorem~\ref{T:directSame}, $\bt$ has exactly two directive words: $\breve\Delta$ and 
$\bar a \bar b\bar c a(b\bar c\bar a)^\omega$. Hence, by Theorem~\ref{T:episturmianQuasiperiods}, $\bt$ has only one quasiperiod, namely $Pal(abc)$. Now, to construct an episturmian word having exactly $k$ quasiperiods for a fixed integer $k\geq 1$, we consider $\breve w\breve\Delta  = \breve w a b c\bar a(b\bar c \bar a)^\omega$ where $\breve w$ is a spinned version of any  $L$-spinned word $w$ with $\Alph(w) \cap \{a,b,c\} = \emptyset$ and $|Pal(w)| = k-1$. Then, by Theorem~\ref{T:episturmianQuasiperiods}, the set of quasiperiods of the episturmian word directed by $\breve w\breve\Delta$ consists of the $k$ words:
\[
  \mu_{\breve w}(Pal(abc))p \quad \mbox{where $p$ is a prefix of $Pal(w)$.}
 \] 
 For example, $d^{k-1}\breve\Delta = d^{k-1} a b c\bar a(b\bar c \bar a)^\omega$ directs an episturmian word with exactly $k$ quasiperiods: $\mu_{d}^{k-1}(Pal(abc))p = d^{k-1}ad^{k-1}bd^{k-1}ad^{k-1}cd^{k-1}ad^{k-1}bd^{k-1}ap$ where $p$ is a prefix of $d^{k-1}$.
\end{xmpl}

\subsection{Characterizations of quasiperiodic episturmian words}

From Theorem~\ref{T:episturmianQuasiperiods}, we immediately obtain the following characterization of quasiperiodic episturmian words. 

\begin{thrm} \label{T:quasi-characterization} An episturmian word is quasiperiodic if and only if 
there exists a spinned word $\breve w$, an $L$-spinned word $v$ and a spinned version $\breve \by$ of an $L$-spinned word $\by$ such that $\bt$ is directed by $\breve wv\breve \by$ with $\Alphit(v) = \Alphit(v\by)$. \medskip
\end{thrm}

The above characterization is clearly not useful when one wants to decide whether or not a given episturmian word is quasiperiodic. In this regard, our normalized directive word plays an important role as it provides a more effective way to decide.

\begin{thrm} \label{T:normalized-quasi-characterization} 
An episturmian word $\bt$ is quasiperiodic if and only if the (unique) normalized directive word of $\bt$ takes the form 
$\breve w a v_1 \bar{a} v_2 \bar{a} \cdots v_{k} \bar{a} v\breve \by$ for some $L$-spinned letter $a$, spinned word $\breve w$, $a$-free $L$-spinned words  $v$, $v_1$, \ldots, $v_{k}$ ($k \geq 0$), and a spinned version $\breve \by$ of an $L$-spinned word $\by$ such that  
$\Alphit(av) = \Alphit(av\by)$.
\end{thrm}
\begin{proof}

First assume that an episturmian word $\bt$ is directed by $\breve w a v_1 \bar{a} v_2 \bar{a} \cdots v_{k} \bar{a} v\breve \by$ (as in the hypotheses). 
Then by Theorem~\ref{T:presentation}, $\bt$ is also directed by $\breve w \bar a {\bar v}_1 \bar{a} {\bar v}_2 \bar{a} \cdots {\bar v}_{k} {a} v\breve \by$. By Theorem~\ref{T:quasi-characterization}, $\bt$ is quasiperiodic.

Now assume that $\bt$ is quasiperiodic. By Theorem~\ref{T:quasi-characterization}, one of its directive words is ${\breve w}_1v' \breve\by$ for spinned words ${\breve w}_1$, $\breve\by$ and an $L$-spinned word $v'$ with $\Alph(v') = \Alph(v'\by)$. If $v'$ contains only one letter $a$, $\breve\by = a^\omega$ and the normalized directive word of $\bt$ ends with $aa^\omega$, then the condition is verified.

Now assume that $v'$ contains at least two (different) letters: let us write $v' = av$. Without loss of generality, we can assume that $\breve\by$ and $\breve{w}_1$ are normalized. 
If $\breve{w}_1$ has no suffix in $\bar a \bar\cA^*$, then 

$\breve{w}_1av\by$ is normalized and of the required form (with $k = 0$).
Otherwise $\breve{w}_1 = \breve w \bar a {\bar v}_1 \bar{a} {\bar v}_2 \bar{a} \cdots {\bar v}_{k}$ for some spinned words $\breve w$, $\breve \by$ and some $a$-free $L$-spinned words  $v_1$, \ldots, $v_{k}$ ($k \geq 0$) such that $\breve w$ has no suffix in $\bar a \bar\cA^*$.
Then the normalized directive word of $\bt$ is $\breve w a v_1 \bar{a} v_2 \bar{a} \cdots v_{k} \bar{a} v\breve \by$.
\end{proof}

\begin{xmpl}
The episturmian words with directive words $(a\bar{b}\bar{c})^\omega$, 
$(a\bar{b}c\bar{a}b\bar{c})^\omega$ or
$(a\bar{b}\bar{a}ca\bar{a}\bar{b}cbc\bar{b})^\omega$ are not quasiperiodic whereas the one with normalized directive word $ab\bar abc\bar c(\bar a bc)^\omega$ is quasiperiodic (it is also directed by $\bar a\bar babc\bar c(\bar a bc)^\omega$).
\end{xmpl}

\begin{rmrk} \label{R:quasi-directive} 
It follows from Theorem~\ref{T:quasi-characterization} that any quasiperiodic episturmian word has at least two directive words. 
Indeed, if $\bt$ is a quasiperiodic episturmian word, then $\bt$ has a directive word of the form $\breve\Delta = \breve w v \breve \by$ where the words $\breve w$, $\breve\by$ are spinned versions of some $L$-spinned words $w$, $\by$ and $v$ is an $L$-spinned word such that $\Alph(v) = \Alph(v\by)$. If $\breve\by$ (and hence $\breve\Delta$) has all spins ultimately $L$, then item~\ref{Ti:2} in part $iii)$ of Theorem~\ref{T:directSame} shows that $\bt$ also has a wavy directive word. 
Now suppose $\breve\by$ does not have all spins ultimately $L$. Then $\breve \by$ must be wavy, and hence contains infinitely many $R$-spinned letters. Choose $x$ to be the left-most $R$-spinned letter in $\breve \by$. Then $\breve\by$ begins with $u\bar x$ for some $L$-spinned word $u$ (possibly empty). Hence, since $x$ occurs in $v$, we see that $\breve\Delta = \breve wv\breve \by$ contains the factor $vu\bar x = v'xv''u\bar x$. Thus $\breve\Delta$ does not satisfy the conditions of Theorem~\ref{T:uniqueDirective} as it contains a factor in $x\cA^*\bar x$, and therefore $\bt$ 
does not have a unique directive word.  Moreover, we easily deduce from Theorems~\ref{T:quasi-characterization} and \ref{T:uniqueDirective} that any episturmian word having a unique directive word is necessarily non-quasiperiodic.
\end{rmrk}

In view of the above remark, one might suspect that an episturmian word is non-quasiperiodic if and only if it has a unique directive word. But this is not true. For example, both $ba(\bar b c \bar a)^\omega$ and $\bar b \bar a b (c\bar  a \bar b)^\omega$ direct the {\em same} non-quasiperiodic episturmian word $\bt$ by Theorems~\ref{T:directSame} and \ref{T:quasi-characterization}. These two spinned infinite words are the only directive words for $\bt$, which might lead one to guess that an episturmian word is non-quasiperiodic if and only if it has finitely many directive words. But again, this is not true. For example, as stated in Example~\ref{ex:epi-quasiperiods},  $adbc\bar d(a\bar b\bar c)^\omega$ and $a\bar d \bar b \bar c d (a\bar b \bar c)^\omega$ are the only two directive words of the {\em quasiperiodic} episturmian word $\mu_{adbc\bar d}(a\br)$ where $\br$ is the Tribonacci word. Moreover there exist non-quasiperiodic episturmian words, such as the one directed by $(ab\bar b\bar c)^\omega$, that have infinitely many directive words (the words $(ab\bar b\bar c)^na\bar b\bar c(ab\bar b\bar c)^\omega$ are pairwise different directive words).

Nevertheless, in the Sturmian case we have:
\begin{prpstn} \label{P:Sturmian-Lyndon-unique} A Sturmian word is non-quasiperiodic if and only if it has a unique directive word.
\end{prpstn}
\begin{proof} First let us suppose by way of contradiction that $\bt$ is a non-quasiperiodic Sturmian word, but has more than one directive word. Then, by Proposition~\ref{P:uniqueDirective-Sturmian}, $\bt$ has infinitely many directive words. Moreover, as $\bt$ is aperiodic, all of the directive words of $\bt$ are spinned versions of the same $\Delta \in \{a,b\}^\omega\setminus(\cA^*a^\omega \cup \cA^*b^\omega) $ by Fact~\ref{F:aperiodic-spinned-versions} and none of these directive words are regular wavy. Hence the normalized directive word of $\bs$ contains $ab$ or $ba$. But then $\bt$ is quasiperiodic by Theorem~\ref{T:quasi-characterization}, a contradiction.

Conversely, by Proposition~\ref{P:uniqueDirective-Sturmian}, if a Sturmian word $\bt$ has a unique directive word, then its (normalized) directive word is regular wavy. Such a word clearly does not fulfill Theorem~\ref{T:quasi-characterization} and so $\bt$ is not quasiperiodic.
\end{proof}

Let us recall that a Sturmian word is non-quasiperiodic if and only if it is a Lyndon word~\cite{fLgR04quas}. The generalization of this aspect will be discussed in Section~\ref{S:episturmianLyndon}.

\section{\label{S:QuasiperiodicEpisturmianMorphisms}Quasiperiodicity and episturmian morphisms}

In this section we draw connections between the results of the previous section and {\em strongly quasiperiodic} morphisms. As in \cite{fLgR07quas}, a morphism $f$ on $\cA$ is called {\it strongly quasiperiodic} (on $\cA$) if for any (possibly non-quasiperiodic) infinite word $\bw$, $f(\bw)$ is quasiperiodic.

\subsection{\label{quasiperiodicity}Strongly quasiperiodic epistandard morphisms}

Quasiperiodicity of epistandard words can also be explained by the strong quasiperiodicity of epistandard morphisms.   

\begin{thrm}\label{T:stronglyQuasEpistandardMorphisms}
Let $v$ be a word over an alphabet $\cA$ containing at least two letters. The epistandard morphism 
$\mu_v$ is strongly quasiperiodic if and only if $\Alphit(v) = \cA$. Moreover $Pal(v)$ is a quasiperiod of $\mu_v(\bw)$ for any infinite word $\bw$.
\end{thrm}

To prove this result, which is a direct consequence of Lemma~\ref{L:coverPart1} (below), we need to consider infinite words covered by several words. We say that a set $X$ of words covers an infinite word $\bw$ if and only if there exist two sequences of words $(p_n)_{n \geq 0}$ and $(z_n)_{n \geq 0}$ such that, for all $n \geq 0$, $p_nz_n$ is a prefix of $\bw$, $z_n \in X$,  $p_0 = \varepsilon$ and $|p_n| < |p_{n+1}| \leq |p_nz_n|$. The last inequalities mean that $p_n$ is a prefix of $p_{n+1}$ which is a prefix of $p_nz_n$, itself a prefix of $\bw$. Once again the sequence $(p_nz_n)_{n\geq 0}$ is called a covering sequence of prefixes of the word $\bw$. Observe that a word is covered by $X$ if and only if it is covered by $X \cup \{\varepsilon\}$.

\medskip

\begin{lmm}\label{L:coverPart1}
For  $\bw$ an infinite word over $\cA$, $y$ a letter, $X$ a subset of $\cA$ and $u$ a word, if $\bw$ is covered by $\{u\}\cup \{ux \mid x \in X \cup \{y\}\}$
then $L_y(\bw)$ is covered by $\{L_y(u)y\} \cup \{L_y(u)yx \mid x \in X \setminus \{y\}\}$.
\end{lmm}

\begin{proof}
If $y \in X$, then $L_y(\bw)$ is covered by $\{L_y(u)\} \cup \{L_y(u)y\} \cup \{L_y(u)yx \mid x \in X\setminus \{y\}\}$, and each occurrence of $L_y(u)$ is followed by the letter $y$. So $L_y(\bw)$ is covered by $\{L_y(u)y\} \cup \{L_y(u)yx \mid x \in X\setminus \{y\}\}$.
If $y \not \in X$, thus $X=X \setminus \{y\}$ and  then $L_y(\bw)$ is covered by $\{L_y(u)\} \cup \{L_y(u)yx \mid x \in X\}$, and similarly each occurrence of $L_y(u)$ is followed by the letter $y$. So $L_y(\bw)$ is covered by $\{L_y(u)y\} \cup \{L_y(u)yx \mid x \in X\}$.
\end{proof}

\medskip
\begin{proof}[Proof of Theorem~$\ref{T:stronglyQuasEpistandardMorphisms}$]
Let $v = v_1\cdots v_{|w|}$ be a word (with each $v_i$ a letter) such that $\Alph(v) = \cA$.
Let $u_{|v|} = \varepsilon$ and, for any $i$ from $|v|$ to $1$, let $u_{i-1}=L_{v_i}(u_i)v_i$: observe that $u_{i-1} = Pal(v_i\cdots v_{|v|})$. 
It is immediate that any infinite word $\bw$ is covered by the set $\cA$, which can be expressed as $\cA = \{\varepsilon\}\cup \{\varepsilon x \mid x \in \cA\} =  \{u_{|v|}\} \cup \{u_{|v|}x \mid x \in \cA\}$. 

By induction using Lemma~\ref{L:coverPart1}, we can state that, for $1 \leq i \leq |v|$, $\mu_{v_i \cdots v_{|v|}}(\bw)$ is covered by $\{u_{i-1}\} \cup \{u_{i-1}x \mid x \in \cA \setminus \{v_i, \ldots, v_{|v|}\}\}$. So, since $u_0=Pal(v)$, $\mu_v(\bw)$ is covered by $\{Pal(v)\} \cup \{Pal(v)x \mid x \in \cA \setminus \{v_1, \ldots, v_{|v|}\}\}$. Therefore, since $\Alph(v) = \cA$, $\mu_v(w)$ is covered by $Pal(v)$, that is, $\bw$ is $Pal(v)$-quasiperiodic. Hence $\mu_v$ is strongly quasiperiodic.

\medskip

Now let us consider the case where $\Alph(v) \neq \cA$. First suppose that $v  = \varepsilon$. Then, since $\cA$ contains at least two letters (and so there exists at least one non-quasiperiodic word over $\cA)$, the morphism $\mu_v$,  which is the identity, is not strongly quasiperiodic.  
Assume now that $v \neq \varepsilon$. Let $a$ be a letter in $\cA \setminus \Alph(v)$ and let $b$ be a letter in $\Alph(v)$.
The word $\mu_v(ab^\omega)$ contains only one occurrence of the letter $a$, and so it is non-quasiperiodic. Thus the morphism $\mu_v$ is not strongly quasiperiodic.
\end{proof}

\begin{rmrk} Note that any infinite word $\Delta$ can be written $\Delta = v\by$ with $\Alph(v) = \Alph(v\by)$. So if $\bt$ (resp.~$\bt'$) is the epistandard word directed by $\Delta$ (resp.~$\by$), then $\bt = \mu_v(\bt')$ and $\bt$ is quasiperiodic by Theorem~\ref{T:stronglyQuasEpistandardMorphisms}.  The same approach allows us to show that if an episturmian word is directed by a spinned infinite word with all spins ultimately $L$, then it is quasiperiodic, as deduced previously (see Remark~\ref{R:ultimately-L}). Note also that Corollary~\ref{Cor:wvy} is a direct consequence of Theorem~\ref{T:stronglyQuasEpistandardMorphisms}.
\end{rmrk}

\subsection{Useful lemmas}

We now introduce useful material for a second proof of Theorem~\ref{T:episturmianQuasiperiods} using morphisms (see Section~\ref{SS:secondProofCharacterizationQuasiperiods}). The next results (that maybe some readers will read when necessary to understand the proof of Theorem~\ref{T:episturmianQuasiperiods}) show how to obtain some quasiperiods (or a covering set) of a word of the form $L_x(\bw)$ or $R_x(\bw)$ for an infinite word $\bw$ (not necessarily episturmian).

\begin{lmm}\label{L:coverPart2} {\em (converse of Lemma~$\ref{L:coverPart1}$)}\\
For an infinite word $\bw$  over $\cA$, $y$ a letter, $X$ a subset of $\cA$ and $u$ a word, if $L_y(\bw)$ is covered by $\{L_y(u)y\} \cup \{L_y(u)yx \mid x \in X \setminus \{y\}\}$, then $\bw$ is covered by $\{u\}\cup \{ux \mid x \in X \cup \{y\}\}$.
\end{lmm}

\begin{proof}
By hypothesis, there exist $(p_n)_{n \geq 0}$, $(u_n)_{n\geq 0}$ such that $p_0=\varepsilon$, $u_n \in \{L_y(u)y\} \cup \{L_y(u)yx \mid x \in X \setminus \{y\}\}$ and for all $n \geq 0$, both $|p_n| < |p_{n+1}| \leq |p_nu_n|$ and $p_nu_n$ is a prefix of $\bw$.

\noindent
For all $n \geq 0$, $p_nL_y(u)y$ is a prefix of $L_y(\bw)$ so there exists $p'_n$ such that $p_n=L_y(p'_n)$.  

\noindent
We consider the following two complementary cases:
\begin{enumerate}
\item
Case $u_n=L_y(u)y$.

The word $p_nL_y(u)y=L_y(p'_nu)y$ is a prefix of $L_y(\bw)$ so $p'_n$ is a prefix of $\bw$. For the prefixes $p_{n+1}$ and $p_nu_n$ of $\bw$, we have $|p_{n+1}| \leq |p_nu_n|$, thus $p_{n+1}$ is a prefix of $p_nu_n$. 

If $|p_{n+1}| < |p_nu_n|$, let $u'_n=u$: $p_{n+1}=L_y(p_{n+1}')$ is a prefix of $p_nL_y(u)=L_y(p'_nu)$. Moreover $p_{n+1}y$ and $p_nL_y(u)y$ are prefixes of $L_y(w)$. So $p_{n+1}'$ is a prefix of $p'_nu$ which itself is a prefix of $w$. 

If $|p_{n+1}| = |p_nu_n|$, then $p_{n+1}=p_nu_n$ and from $p_{n+1}u_{n+1}$ prefix of $\bw$, we know that $p_{n+1}y$ is a prefix of $w$, hence $p'_nuy=p_{n+1}'$ is a prefix of $\bw$. In  this case let $u'_n=uy$.

\item
Case $u_n=L_y(u)yx$ for some $x \in X \setminus \{y\}$. In this case (as previously), we can see that with $u'_n=ux$, $p_{n+1}'$ is a prefix of $p'_nu'_n$ itself a prefix of $\bw$. 
\end{enumerate}

\medskip
We have proved that the sequence $(p'_nu'_n)_{n \geq 0}$ is a covering sequence of prefixes of $\bw$. Thus the word $\bw$ is covered by $\{u\} \cup \{ux \mid x \in X \cup \{y\}\}$.
\end{proof}

\begin{lmm}\label{L:new1}
{\rm (generalization of \cite[Lem.~5.5]{fLgR07quas})}\\
Let $\bw$ be an infinite word starting with a letter $x$. For any letter $y \neq x$, the word $R_y(\bw)$ is quasiperiodic if and only if $\bw$ is quasiperiodic. 

\noindent
Moreover if $q$ is a quasiperiod of $R_y(\bw)$ then two cases are possible:
\begin{enumerate}
\item $q = R_y(q')$ with $q'$ a quasiperiod of $\bw$;
\item $qy = R_y(q'z)$ with $z \neq a$ a letter such that both $q'$ and $q'z$ are quasiperiods of $\bw$.
\end{enumerate}
\end{lmm}

\begin{proof}
Assume first that $\bw$ is quasiperiodic. By Fact~\ref{rnew4}, $R_y(\bw)$ is quasiperiodic. More precisely, if $v$ is a quasiperiod of $\bw$, then $R_y(v)$ is a quasiperiod of $R_y(\bw)$.

\medskip

Assume now that $R_y(\bw)$ is quasiperiodic and let $q$ be one of its quasiperiods. 
By hypothesis, we know that $q$ starts with $x$ (as $\bw$). 

A first case is that $q = R_y(v)$ (for a word $v$) which is equivalent to the fact that $q$ ends with the letter $y$. Let $(p_n)_{n \geq 0}$ be a sequence of words such that $(p_nq)_{n \geq 0}$ is a covering sequence of prefixes of $R_y(\bw)$. Let $n \geq 0$. From the fact that $q$ starts with $x$, we deduce that there exists a word $p_n'$ such that $p_n =R_y(p_n')$. Then we can see that $p_n'v$ is a prefix of $\bw$. Moreover the inequality $|p_n| < |p_{n+1}| \leq |p_nq|$ implies  $|p_n'| < |p_{n+1}'| \leq |p_n'v|$. So $(p_n'v)_{n \geq 0}$ is a covering sequence of prefixes of $\bw$: $v$ is a quasiperiod of $\bw$.

Now assume that $q$ ends with a letter $z$ different from $y$. Then $q = R_y(v)z$ for a word $v$ and each occurrence of $q$ is followed by the letter $y$. Let $(p_n)_{n \geq 0}$ be a sequence of words such that $(p_nq)_{n \geq 0}$ is a covering sequence of prefixes of $R_y(\bw)$. Let $n \geq 0$. From the fact that $q$ starts with $x$, we deduce that there exists a word $p_n'$ such that $p_n =R_y(p_n')$. Then we can see that $p_n'v$ is a prefix of $\bw$. Now since $q$ does not end with $y$, the inequality $|p_n| < |p_{n+1}| \leq |p_nq|$ is actually 
$|p_n| < |p_{n+1}| \leq |p_nR_y(v)|$. And so $(p_n'v)_{n \geq 0}$ is a covering sequence of prefixes of $\bw$; thus $v$ is a quasiperiod of $\bw$. Since each occurrence of $R_y(v)$ is followed by $y$, $(p_n'vz)_{n \geq 0}$ is also a covering sequence of prefixes of $\bw$; whence $vz$ is also a quasiperiod of $\bw$.
\end{proof}

\begin{lmm}\label{L:new2} 
Let $x, y$ be two different letters, let $\bw$ be an infinite word, and let $u$ be a finite word such that $L_x(\bw)$ is $uy$-quasiperiodic. Then there exists a word $v$ such that $\bw$ is $vy$-quasiperiodic, $uy = L_x(vy)$, and $|uy| > |vy|$.
\end{lmm}

\begin{proof}
Let $x, y, u, \bw$ be as in the lemma. Since $uy$ is a prefix of $L_x(w)$ and since $x \neq y$, there exists a word $v$ such that $uy = L_x(vy)$. We have $|uy| > |vy|$.

Let $(p_n)_{n \geq 0}$ be such that  $(p_nuy)_{n \geq 0}$ is a covering sequence of prefixes of $L_x(\bw)$. Since $uy$ is a prefix of $L_x(\bw)$, $x$ is the first letter of $u$, and so from $p_nx$ prefix of $L_x(\bw)$, there exists for any $n \geq 0$ a word $p_n'$ such that $p_n = L_x(p_n')$.  From the prefix $L_x(p_n'vy)$ of $L_x(\bw)$, we deduce that $p_n'vy$ is a prefix of $\bw$. The inequality 
$|p_n| < |p_{n+1}| \leq |p_nuy|$, that is 
$|L_x(p_n')| < |L_x(p_{n+1}')| \leq |L_x(p_n'vy)|$,  implies
that $|p_n'| < |p_{n+1}'| \leq |p_n'vy|$. Hence $vy$ covers $\bw$. 
\end{proof}

\begin{rmrk}
Lemmas~\ref{L:new1} and \ref{L:new2} are not true when $x = y$. For instance $ab^\omega$ is not quasiperiodic whereas $R_a(ab^\omega) = a(ba)^\omega$ is quasiperiodic.
Moreover the word $L_a(b^\omega) = (ab)^\omega$ is $aba$-quasiperiodic whereas $b^\omega$ has no quasiperiod ending with $a$.
\end{rmrk}

\begin{lmm}\label{L:new3} 
Suppose $\bt$ is an episturmian word starting with a letter $z$. Let $x$ be a letter different from $z$ and let $u$ be a non-empty word.  Then there exists a set  $\cB$  of letters such that the word $R_x(\bt)$ is covered by $\{u\} \cup \{u\}\cB$ if and only if this word is covered by $\{u, ux\}$.
\end{lmm}

\begin{proof}
The if part is immediate. Assume that $R_x(\bt)$ is covered by $\{u\} \cup \{u\}\cB$. There is nothing to prove if $\cB \subseteq \{x\}$;  hence we assume that $\cB$ contains a letter $y$ different from $x$. Since $\bt$ starts with $z$ and $z\neq x$, the word $R_x(\bt)$ and its prefix $u$ start with $z$. 
If the word $uy$ is a factor of $R_x(\bt)$, the word $u$ ends with the letter $x$. 
Let $(p_n)_{n \geq 0}$ and $(u_n)_{n \geq 0}$ such that $(p_nu_n)_{n \geq 0}$ is a covering sequence of prefixes of $R_x(\bt)$ and for all $n \geq 0$, $u_n \in \{u\} \cup \{u\}\cB$. 
Let $n$ be any integer such that $u_n = uy$. Since $y \neq x$ and $z\neq x$, the word $yz$ is not a factor of $R_x(\bt)$. Since $|p_{n+1}| \leq |p_nu_n|$, since $p_{n+1}$ and $p_nu_n$ are prefixes of $R_x(\bt)$, and since $u_{n+1}$ starts with $z$, necessarily $|p_{n+1}| < |p_nu_n| = |p_nuy|$. We deduce from what precedes that $R_x(\bt)$ is covered by $\{x\} \cup \{u\}(\cB\setminus\{y\})$. Hence considering successively all letters of $\cB$ different from $x$, it follows that $R_x(\bt)$ is covered by $\{u, ux\}$.
\end{proof}

\begin{lmm}
\label{L:new4}
Let $\bt$ be an episturmian word starting with a letter $z$. Let $x$ be a letter different from $z$ and let $q$ be a word. If the word $R_x(\bt)$ is covered by $\{q, qx\}$ then $L_x(\bt)$ is covered by $xq$. 
\end{lmm}

\begin{proof}
If $q$ is empty, the result is immediate since in this case $\bt = x^\omega$. If $q$ ends with a letter different from $x$, then each occurrence of $q$ is followed by an $x$ and so $R_x(\bt)$ is covered by $qx$. Hence we can assume without loss of generality that $q$ ends with $x$. Then $q = R_x(v)$ and $qx = R_x(vx)$ for some word $v$. By the same technique used in the proof of Lemma~\ref{L:new1}, we can see that $\bt$ is covered by $\{v, vx\}$. This implies that $L_x(\bt)$ is covered by $\{L_x(v), L_x(v)x\}$ with all the occurrences of $L_x(v)$ followed by an $x$. So $L_x(\bt)$ is covered by $L_x(v)x$. It is well-known that $L_x(v)x = xR_x(v) = xq$. 
\end{proof}

\begin{lmm}
\label{L:new5}
Suppose $\bt$ is an infinite word (not necessarily episturmian) covered by the set $\{q, qa\}$ and by the set $\{q, qc\}$ for some word $q$ and two different letters $a$ and $c$. Then $\bt$ is $q$-quasiperiodic.
\end{lmm}

\begin{proof}
Let $(p_n)_{n \geq 0}$ be prefixes of $\bt$ and $(q_n)_{n\geq 0}$ be words belonging to $\{q, qc\}$ such that $(p_nq_n)_{n \geq 0}$ is a covering sequence of $\bt$. Moreover let us assume that there is no word $p$ except the elements of $(p_n)_{n\geq 0}$ such that $pq$ is a prefix of $\bt$. Let $n$ be any integer such that $q_n = qc$. Since $\bt$ is also covered by $\{q, qa\}$ and since all the words $p$ such that $pq$ is a prefix of $\bt$ belong to $\{p_n \mid n \geq 0\}$, we necessarily have $|p_{n+1}| \leq |pq|$ (that is $|p_{n+1}| \neq |p_nqc|$). The sequence $(p_nq)$ is a covering sequence of $\bt$; hence $\bt$ is $q$-quasiperiodic.
\end{proof}

\subsection{\label{SS:secondProofCharacterizationQuasiperiods}A second proof of Theorem~\ref{T:episturmianQuasiperiods}} 

We now give a second proof of Theorem~\ref{T:episturmianQuasiperiods} which provides further insight into the connection between quasiperiodicity and morphisms.

\medskip
For any spinned word $\breve w$, let us denote by $q_{\breve{w}}$ the word $S_{\breve{w}}^{-1}Pal(w)$ appearing in Theorem~\ref{T:episturmianQuasiperiods}. Then, for any spinned word $\breve v$ and letter $a$, we have:
\begin{itemize}
\item $q_\varepsilon = \varepsilon$; 

\item $q_{a\breve v} = L_a(q_{\breve v})a$;

\item $q_{\bar a \breve v} = R_a(q_{\breve v})$.  
\end{itemize}

Indeed by formula~\eqref{eq:f-shift}, we have $S_{a\breve{v}} = L_a(S_{\breve{v}})$ and
$S_{\bar a\breve{v}} = L_a(S_{\breve{v}})a = aR_a(S_{\breve{v}})$. Hence $q_{a\breve{v}} = S_{a\breve{v}}^{-1}L_a(Pal(v))a = (L_a(S_{\breve v}))^{-1}L_a(Pal(v))a = L_a(S_{\breve v}^{-1}Pal(v))a = L_a(q_{\breve v})a$
and, since $L_a(Pal(v))a = aR_a(Pal(v))$, $q_{\bar a\breve{v}} = S_{\bar a\breve{v}}^{-1}(aR_a(Pal(v))) = (aR_a(S_{\breve v}))^{-1}(aR_a(Pal(\breve v))) = R_a(S_{\breve v}^{-1}(Pal(\breve v))$.

These formulae, which define recursively the word $q_{\breve{w}}$, will be helpful in the following proof.

\medskip

First assume that $\bt$ is an episturmian word directed by $\breve wv\breve \by$ where $\breve w$ is a spinned word, $\breve\by$ is a spinned version of an $L$-spinned infinite word $\by$ and $v$ is an $L$-spinned word such that $\Alph(v) = \Alph(v\by)$. 
Then, by Theorem~\ref{T:stronglyQuasEpistandardMorphisms}, $Pal(v)$ is a  quasiperiod of $\mu_v(\bt')$ where $\bt'$ is the episturmian word directed by $\breve \by$. 
Now we prove by induction on $|\breve w|$ that $\mu_{\breve w}(Pal(v))p$ is a quasiperiod of $\bt$ for any prefix $p$ of $q_{\breve w}$. More precisely we prove that $\mu_{\breve w}(Pal(v))$ is a quasiperiod of $\bt = \mu_{\breve w}(\mu_{v}(\bt'))$ (this is also a consequence of Fact~\ref{rnew4}) and each occurrence of $\mu_{\breve w}(Pal(v))$ in $\bt$ is followed by the word $q_{\breve w}$ (which is a direct consequence of the above inductive formulae). Observe that the previous fact obviously holds if $\breve{w} = \varepsilon$. When it holds, let $(p_n)_{n \geq 0}$ be a sequence of prefixes such that $|p_n| \leq |p_{n+1}| \leq |p_n\mu_{\breve w}(Pal(v))|$ and $p_n\mu_{\breve w}(Pal(v))q_{\breve w}$ is a prefix of $\bt$. Then
$L_a(p_n\mu_{\breve w}(Pal(v))q_{\breve w}) = L_a(p_n) \mu_{a\breve w}(Pal(v))L_a(q_{\breve w})$ and is followed as a prefix of $L_a(\bt)$ by $a$. Note that the word $L_a(p_n) \mu_{a\breve w}(Pal(v))L_a(q_{\breve w})a = 
L_a(p_n) \mu_{a\breve w}(Pal(v))q_{a\breve w}$ is a prefix of $L_a(\bt)$. Moreover we can verify that
$|L_a(p_n)| \leq |L_a(p_{n+1})| \leq |L_a(p_n)\mu_{a\breve w}(Pal(v))|$ since $p_n$ is a prefix of $p_{n+1}$ itself a prefix of $p_n\mu_{\breve w}(Pal(v))$. 
Similarly $R_a(p_n\mu_{\breve w}(Pal(v))q_{\breve w}) = R_a(p_n)\mu_{\bar a\breve w}(Pal(v))R_a(q_{\breve w}) = R_a(p_n)\mu_{\bar a \breve w}(Pal(v))q_{\bar a\breve w}$ is a prefix of $R_a(\bt)$ and $|R_a(p_n)| \leq |R_a(p_{n+1})| \leq |R_a(p_n)\mu_{\bar a\breve w}(Pal(v))|$.

\bigskip

To end the proof of Theorem~\ref{T:episturmianQuasiperiods}, we prove by induction on $|q|$ that if $q$ is the quasiperiod of an episturmian word $\bt$ then (at least) one directive word of $\bt$ can be written $\Delta = \breve w v \breve \by$ where $\breve w$ is a spinned word,  $\breve\by$ is a spinned version of an $L$-spinned infinite word $\by$ and $v$ is an $L$-spinned word such that $\Alph(v) = \Alph(v\by)$. Moreover we have $q = \mu_{\breve{w}}(Pal(v))p$ with $p$ a prefix of $q_{\breve{w}} = S_{\breve{w}}^{-1}Pal(w)$.

We observe that $|q| \geq 1$ since $q$ covers the infinite word $\bt$.

When $|q| = 1$, for a letter $a$, $q = a$ and the word $\bt = a^\omega$ is directed by $a^\omega$ which is of the form $\breve w v a^\omega$ with $\breve w = \varepsilon$ and $v = a$. Then $q_{\breve{w}}  = \varepsilon$, and taking $p = \varepsilon$, $q = \mu_{\breve{w}}(Pal(v))p$. 

\medskip

Assume from now on that $|q| \geq 2$.
Without loss of generality, we assume that $\Delta = {\breve x}_1 {\breve x}_2 \cdots$ is the normalized directive word of $\bt$, and let $(\bt^{(n)})_{n \geq 0}$ be the infinite sequence of episturmian words associated to $\bt$ and $\Delta$ in Theorem~\ref{T:episturmian}.

\medskip
Let us first consider: \\
\textit{Case $\bt = L_a(\bt^{(1)})$ (that is, ${\breve x}_1 = a$ for a letter $a$).} 

Assume that $q$ does not end with the letter $a$. By Lemma~\ref{L:new2}, 
$q = L_a(q_1)$ for a quasiperiod $q_1$ of $\bt^{(1)}$ such that $|q_1| < |q|$. 
By the induction hypothesis, $\bt^{(1)}$ has a directive word of the form $\breve{w}v\breve{\by}$ 
with $\breve{w}$, $v$, $\breve{\by}$ as in the theorem, and $q_1 = \mu_{\breve{w}}(Pal(v))p$ for a prefix $p$ of $q_{\breve{w}}$. We have $q = L_a(\mu_{\breve{w}}(Pal(v))p) = \mu_{a\breve{w}}(Pal(v))L_a(p)$. Observing that $\bt$ is directed by $(a\breve{w})v\breve{\by}$ and that $L_a(p)$ is a prefix of $L_a(q_{\breve{w}})a = q_{a\breve{w}}$, we get the result for $q$ and $\bt$.

Now assume that $q$ ends with the letter $a$, that is $q = L_a(q_1)a$ for a word $q_1$. By Lemma~\ref{L:coverPart2},  $\bt^{(1)}$ is covered by the set $\{q_1, q_1a\}$. 

Possibly $q_1a$ is a quasiperiod of $\bt^{(1)}$. In this case, $q = L_a(q_1a)$. Observe that $|q| = |L_a(q_1a)| \geq |q_1a|$. The equality holds only when $a$ is the only letter occurring in $q_1$ so in $q$. In this case, taking $\breve w = a^{|q|-1}$ and $v = a$, the word $\bt$ is directed by $\breve wva^\omega$ and we have $q = aa^{|q|-1} = Pal(a)q_{a^{|q|-1}} = \mu_{a^{|q|-1}}(Pal(a))q_{\breve w} = \mu_{\breve w}(Pal(v))q_{\breve w}$: the induction result holds here. 
When $|q| = |L_a(q_1a)| > |q_1a|$, the proof ends, using the induction hypothesis,  as in the previous case where $q$ did not end by $a$ (the only difference is that $q$ ends with the word $L_a(p)a$ but $L_a(p)a$ is still a prefix of $q_{a\breve{w}}$).

It is also possible that $q_1$ is a quasiperiod of $\bt^{(1)}$, but in this case we can once again  conclude using the induction hypothesis. 

Now we assume that neither $q_1$ nor $q_1a$ is a quasiperiod of $\bt^{(1)}$. This implies in particular that $\bt^{(1)}$ has a factor $q_1b$ for a letter $b$ different from $a$. So the following property holds for  $n =1$ (and $a_1 = a$):
\begin{quote}
\textit{Prop(n)}: $\bt = L_{a_1}L_{a_2}\cdots L_{a_n}(\bt^{(n)})$ (that is, $\breve{x}_i = a_i$, for $1 \leq i \leq n$), $\bt^{(n)}$ is covered by the set $\{q_n\} \cup \{q_na_i \mid i = 1, \ldots, n\}$ with $q_0 = q$, and for $i = 1, \ldots, n$, $q_{i-1} = L_{a_i}(q_i)a_i$. Moreover $q_na$ and $q_nb$ are both factors of $\bt^{(n)}$. 
\end{quote}
Let us assume that Property~\textit{Prop(n)} holds for some $n \geq 1$ with $q_n \neq \varepsilon$ and $\bt^{(n)} = L_{a_{n+1}}(\bt^{(n+1)})$ (that is ${\breve x}_{n+1} = a_{n+1}$). From $a \neq b$ (so $a \neq a_{n+1}$ or $b \neq a_{n+1}$), we deduce that $q_nc$ is a factor of $\bt^{(n)}$ for a letter $c \neq a_{n+1}$ and so $q_n$ must end with the letter $a_{n+1}$, that is, $q_n = L_{a_{n+1}}(q_{n+1})a_{n+1}$ for some word $q_{n+1}$. By Lemma~\ref{L:coverPart2},  $q_{n+1}$ is covered by the set $\{q_{n+1}\} \cup \{q_{n+1}a_i \mid i = 1, \ldots, n, n+1\}$. Moreover it is quite immediate that $q_{n+1}a$ and $q_{n+1}b$ are both factors of $\bt^{(n+1)}$. Hence \textit{Prop(n+1)} holds. 

Observing that for any $i \geq 1$ we have  $|q_{i+1}| < |q_i|$, we deduce that one of the two following cases holds:
\begin{enumerate} 
\item \textit{There exists an integer $n \geq 1$ such that Property~Prop(n) holds with $q_n = \varepsilon$.} 
In this case, we verify that $q = Pal(a_1\cdots a_n)$. Let $\breve w = \varepsilon$, $v = a_1\cdots a_n$ and let $\breve \by$ be any directive word of the episturmian word $\bt^{(n)}$. Then $\bt$ is directed by $\breve w v \breve \by$ and $q = \mu_{\breve{w}}(Pal(v))$. Since $\bt^{(n)}$ is covered by the letters of $v$, we deduce that $\Alph(v) = \Alph(v\by)$. Hence the induction result holds.

\item \textit{There exists an integer $n \geq 1$ and a letter $c$ such that Property~Prop(n) holds with $q_n \neq \varepsilon$ and $\bt^{(n)}  = R_c(\bt^{(n+1)})$} (that is ${\breve x}_{n+1} = \bar c$). By Theorem~\ref{T:normalisation}, since $\Delta$ is normalized, the word $\bt^{(n+1)}$ does not start with the letter $c$ (otherwise the infinite word $\breve{x}_{n+1}\breve{x}_{n+2}\cdots$ starts with a factor in $\bar c \bar\cA c$).  Hence the first letter of $q_n$ is not $c$ whereas its last letter is $c$ since $a \neq b$ and both $q_na$ and $q_nb$ are factors of $R_c(\bt^{(n+1)})$. By Lemma~\ref{L:new3}, we deduce that $R_c(\bt^{(n+1)})$ is covered by $\{q_n, q_nc\}$. Since for $i = 1, \ldots, n$, $q_{i-1} = L_{a_i}(q_i)a_i$, we can inductively verify that the words $\bt^{(i)}$ ($i = n, \ldots, 1$) are covered by $\{q_i, q_ic\}$. In particular $\bt^{(1)}$ is covered by $\{q_1, q_1c\}$. 
By Lemma~\ref{L:new5}, since we have assumed that $\bt^{(1)}$ is not $q_1$-quasiperiodic, we have $c = a$. So $R_a(\bt^{(n+1)})$ is covered by $\{q_n, q_na\}$. By Lemma~\ref{L:new4}, $L_a(\bt^{(n+1)})$ is $aq_n$-quasiperiodic. 
Observe that $\bt = L_{a_1}L_{a_2} \cdots L_{a_n}R_{a_1}(\bt^{(n+1)}) = R_{a_1}R_{a_2} \cdots R_{a_n}L_{a}(\bt^{(n+1)})$, and that $|q_n| < |q|$. By induction hypothesis, $L_a(\bt^{(n+1)})$ is directed by a word $\breve w' v \breve\by$ for a finite spinned word $\breve w'$, a finite $L$-spinned word $v$ and a spinned version $\breve\by$ of an infinite $L$-spinned word $\by$ such that $\Alph(v) = \Alph(v\by)$, and $aq_n = \mu_{\breve w'}(Pal(v))p'$ with $p'$ a prefix of $q_{\breve w'}$. Thus we have proved that the word $\bt$ is directed by $\breve w v\breve\by$ where $\breve w = \bar{a}_1 \bar{a}_2 \cdots \bar{a}_n \breve w'$, and $\mu_{\bar{a}_1\bar{a}_2\cdots \bar{a}_n}(qa_n) = \mu_{\breve w}(Pal(v))p'$.
To end the proof of the current case we have to prove that $q = \mu_{\breve w}(Pal(v))p'$. First assume that $a_i = a$ for some $i$ between $2$ and $n$. We know that the word $\bt^{(i)}$ is covered by $\{q_i, q_ia\}$ and $\bt^{(i-1)} = L_{a_i}(\bt^{(i)}) = L_{a}(\bt^{(i)})$. Hence by Lemma~\ref{L:coverPart1}, $\bt^{(i-1)}$ is covered by $L_a(q_i)a = q_{i-1}$. Using Fact~\ref{rnew4}, we can deduce that 
$\bt^{(1)}$ is $q_1a$-quasiperiodic, a contradiction. Hence for $i =2, \ldots, n$, $a_i \neq a$. By induction we can prove that $\mu_{\bar{a}_i\bar{a}_{i+1}\cdots \bar{a}_n}(aq_n)= a q_{i-1}$ for $i = n+1, \ldots, 2$. Indeed if $\mu_{\bar{a}_{i+1}\cdots \bar{a}_n}(aq_n)= a q_{i}$ then $\mu_{\bar{a}_i\bar{a}_{i+1}\cdots \bar{a}_n}(aq_n)=$ $\mu_{\bar{a}_i}(aq_i) =$ $R_{a_i}(aq_i) =$ $aa_i R_{a_i}(q_i) =$ $aL_{a_i}(q_i){a_i} =$ $ aq_{i-1}$. Now 
$\mu_{\bar{a}_1\bar{a}_{2}\cdots \bar{a}_n}(aq_n)=$ $R_{a_1}(aq_1) =$ $aR_a(q_1) = L_a(q_1)a = q$.

\end{enumerate}

Now we come to:

\noindent
\textit{Case $\bt = R_a(\bt^{(1)})$.}  
By Theorem~\ref{T:normalisation}, since $\Delta$ is the normalized directive word of $\bt$, the word $\bt$ does not start with the letter $a$.  By Lemma~\ref{L:new1}, the word $\bt^{(1)}$ is quasiperiodic and more precisely there exists a quasiperiod $q_1$ of $\bt^{(1)}$ such that 
$q = R_a(q_1)$  or $q = R_a(q_1)z$ with $z \neq a$: in this last case, $q_1z$ is also a quasiperiod of $\bt^{(1)}$. Since $q$ does not start with $a$, $|q_1| < |q|$.
By induction hypothesis, $\bt^{(1)}$ has a directed word of the form $\breve{w}v\breve{\by}$ 
with $\breve{w}$, $v$, $\breve{\by}$ as in the theorem, and $q_1 = \mu_{\breve{w}}(Pal(v))p$ for a prefix $p$ of $q_{\breve{w}}$. When $q = R_a(q_1)$, since $R_a(\mu_{\breve{w}}) = \mu_{\bar a \breve w}$ and $R_a(p)$ is a prefix of $R_a(q_{\breve{w}}) = q_{\bar a\breve w}$, $q = \mu_{\bar a \breve w}(Pal(v))R_a(p)$ and $\bt$ (which is directed by $\bar a \breve w v \breve \by$) verify the induction result. When $q = R_a(q_1)z$ with $z \neq a$, since $q_1z$ is also a quasiperiod of $\bt^{(1)}$, we deduce that $pz$ is a prefix of $q_{\breve w}$ and so $R_a(p)z$ is a prefix of $R_a(p)za$ a prefix of $q_{\bar a\breve w}$, which allows us to conclude once again that $q = \mu_{\bar a \breve w}(Pal(v))R_a(p)z$ and $\bt$ (which is directed by $\bar a \breve w v \breve \by$) verify the induction result. 

This ends the second proof of Theorem~\ref{T:episturmianQuasiperiods}.

\subsection{Strongly quasiperiodic episturmian morphisms} \label{SS:stongly-quasi-morphisms}

The aim of this section is to characterize all the episturmian morphisms that are strongly quasiperiodic, i.e., the episturmian morphisms that map any infinite word onto a quasiperiodic word. Our characterization (Theorem~\ref{T:quasi-morphisms}) generalizes Theorem~\ref{T:stronglyQuasEpistandardMorphisms} to all episturmian morphisms.

Looking at Theorem~\ref{T:stronglyQuasEpistandardMorphisms} and Theorem~\ref{T:episturmianQuasiperiods}, one might guess that an episturmian morphism $\mu_{\breve u}$ is strongly quasiperiodic if and only if $\breve u$ has an $L$-spinned factor $v$. We will see that it is not the case. In particular, there exist $R$-spinned words $\breve u$ such that $\mu_{\breve u}$ is strongly quasiperiodic, as shown by the following result.

\begin{lmm}
\label{L:reversible2}
Let $\breve u$ be a finite spinned word over $\cA \cup \bar \cA$ and let $a$ be a letter in $\cA$. If there exist $L$-spinned words $v$, $y$ and a spinned word $\breve w$ such that $\breve u = \breve w \bar a \bar v \bar y$ and $\Alphit(v)=\cA \setminus \{a\}$, then for any infinite word $\bt$, $\mu_{\breve u}(a\bt)$ is quasiperiodic.
\end{lmm}

Let us mention that for instance the word $\bar a \bar b \bar c \bar a \bar b$ verifies the condition of the previous lemma for each of the letters $a$, $b$ and $c$, so that $\mu_{\bar a \bar b \bar c \bar a \bar b}$ is strongly quasiperiodic over $\{a, b, c\}$.

Lemma~\ref{L:reversible2} is a consequence of the following one.

\begin{lmm}\label{L:pal}
Suppose $a$ is an $L$-spinned letter, $x$ is an $L$-spinned word, $v$ is an $a$-free $L$-spinned word, and $\bw$ is an infinite word over $\cA$. Then the words $\mu_{av}(xa)$, $\mu_{\bar a \bar v}(ax)$, and $\mu_{\bar a \bar v}(a\bw)$ are $Pal(av)$-quasiperiodic.
\end{lmm}

\begin{proof} ~
\begin{enumerate}
\item[(1)]
By Theorem~\ref{T:stronglyQuasEpistandardMorphisms}, the word $\mu_{av}(xa^\omega)$ is $Pal(av)$-quasiperiodic where $Pal(av) = \mu_{av}(a)$ since $Pal(ava) = \mu_{av}(a)Pal(av) = Pal(av)Pal(av)$ by formulae~\eqref{eq:formula(3)} and \eqref{eq:u_{i+1}}. 
Consequently, since $\mu_{av}(xa)$ ends with $\mu_{av}(a)$, it is $Pal(av)$-quasiperiodic.

\item[(2)] By formulae~\eqref{eq:f-shift} and \eqref{eq:u_n2},
$\mu_{\bar a \bar v}(xa) = S_{\bar a\bar v}^{-1}\mu_{av}(xa)S_{\bar a\bar v}$ and $S_{\bar a\bar v} = Pal(av)$. 
So $Pal(av)\mu_{\bar a \bar v}(ax) =\mu_{av}(ax)\mu_{av}(a)=\mu_{av}(axa)=\mu_{av}(a)\mu_{av}(xa)=Pal(av)\mu_{av}(xa)$, that is {$\mu_{\bar a \bar v} (ax)=\mu_{av}(xa)$}. 
It follows from $(1)$ that $\mu_{\bar a \bar v}(ax)$ is $Pal(av)$-quasiperiodic.

\item[(3)] From (2), we deduce that $\mu_{\bar a \bar v}(a\bw)$ has infinitely many quasiperiodic prefixes with quasiperiod $Pal(av)$. Hence $\mu_{\bar a \bar v}(a\bw)$ is $Pal(av)$-quasiperiodic.
\end{enumerate}
\end{proof}

\begin{proof}[Proof of Lemma~$\ref{L:reversible2}$]
Let $\breve u$, $a$, $\breve w$, $v$, $y$ be as in the hypotheses of Lemma~\ref{L:reversible2}.
Let $\bt$ be an infinite word over $\cA$. The word $\mu_{\bar y}(a\bt)$ starts with $a$. By Lemma~\ref{L:pal}, $\mu_{\bar a \bar v}(\mu_{\bar y}(\bt))$ is quasiperiodic and so $\mu_{\breve u}(a\bt) = \mu_{\breve w \bar a \bar v \bar y}(a\bt)$ is quasiperiodic.
\end{proof}

\medskip

The following remark will be useful several times (see the proof of Theorem~\ref{T:normalized-quasi-characterization} for more details):

\begin{rmrk}\label{R:decomp}
An infinite word over $\cA \cup \bar \cA$ has a decomposition $\breve w v \breve \by$ with $\Alph(v) = \Alph(v\breve \by)$ if and only if its normalized decomposition can be written in the form $\breve w' a v_1 \bar a v_2 \cdots \bar a v_k \breve \by'$ with $\Alph(av_i) = \Alph(av_ia v_{i+1} \cdots a v_k \breve \by')$  for some $1 \leq i \leq k$.
\end{rmrk}

Now we state our characterization of strongly quasiperiodic episturmian morphisms.

\begin{thrm}\label{T:quasi-morphisms}
Let $\cA$ be an alphabet containing at least three letters. 
An episturmian morphism is strongly quasiperiodic on $\cA$ if and only if its normalized directive word $\breve u$ verifies one of the following three conditions:
\begin{itemize}
\item[i)]
$\breve u=\breve w a v_1 \bar a v_2 \bar a \cdots v_k \bar a v \breve {y}$ 

where $a$ is a letter of $\cA$ (with spin $L$), $\breve w$, $\breve y$ are spinned words, $v_1, \ldots, v_k$ ($k \geq 0$) are  $a$-free $L$-spinned words, and $v$ is an $L$-spinned word such that $\Alphit(av)=\Alphit(av\breve y)$.
\item[ii)]
For any letter $a$ in $\cA$, $\breve u= \breve w \bar a \bar v \bar y$ for some spinned word $\breve w$ and $L$-spinned words $v$, $y$ such that $\Alphit(v)=\cA \setminus \{a\}$.
\item[iii)] $\breve u$ verifies case $ii)$ for all letters in $\cA$ except for one letter $a \in \cA$ such that
$\breve u = \breve w v \bar a \bar y$ where $v$ and $y$ are $L$-spinned words verifying $\Alphit(v)=\cA \setminus \{a\}$.
\end{itemize}
\end{thrm}

Before proving this theorem, let us observe that we do not include in this result the Sturmian case. Indeed, Theorem~\ref{T:quasi-morphisms} is no longer valid when $n \leq 2$. For instance the morphism $\mu_{\bar a\bar b}$ ($a \mapsto aba$, $b \mapsto ba$) is strongly quasiperiodic but its normalized directive word fulfills none of the above conditions. A complete description of strongly quasiperiodic Sturmian morphisms is provided in~\cite{fLgR07quas}.

\begin{proof}
Let $\breve u$ be the normalized directive word (of the morphism $\mu_{\breve u}$).

\medskip
If $\breve u$ verifies $i)$, then by Theorem~\ref{T:presentation}, $a v_1 \bar a v_2 \bar a \cdots v_k \bar a v \equiv \bar a \bar {v}_1 \bar a \bar{v}_2 \bar a \cdots \bar{v}_k a v$. Moreover from Theorem~\ref{T:stronglyQuasEpistandardMorphisms}, the epistandard morphism $\mu_{av}$ is strongly quasiperiodic. Thus $\mu_{\breve u}$ is strongly quasiperiodic on $\cA$.

\medskip
If $\breve u$ verifies $ii)$, then by Lemma~\ref{L:reversible2}, $\mu_{\breve u}$ is strongly quasiperiodic over $\cA$.

\medskip
If $\breve u$ verifies $iii)$, then Lemma~\ref{L:reversible2}, $\mu_{\breve u}(\bt)$ is quasiperiodic for any word $\bt$ that does not start with~$a$. 

Let us decompose $y = y_0 a y_1 \cdots a y_k$ where $k$ is the number of occurrences of $a$ in $y$.
In the proof of Lemma~\ref{L:pal}, we have seen that $\mu_{\bar a \bar y_k} (ax)=\mu_{ay_k}(xa)$ for any    finite word $x$. This formula naturally extends to any infinite word $\bt$,
$\mu_{\bar a \bar y_k}(a\bt) = \mu_{ay_k}(\bt)$.
Hence 
$\mu_{\breve u}(a\bt) =$
$\mu_{\breve{w}v\bar{a}\bar{y}_0\bar{a}\bar{y}_1\cdots \bar{y}_{k-1}ay_k}(\bt)$. 
By Theorem~\ref{T:presentation}, $\mu_{\bar{a}\bar{y}_ia} = \mu_{ay_i\bar{a}}$ for each $i$, so that  $\mu_{\breve u}(a\bt) =
\mu_{\breve{w}v{a}{y}_0\bar{a}\bar{y}_1\bar{a}\cdots {y}_{k-1}\bar{a}y_k}(\bt)$.
From Theorem~\ref{T:stronglyQuasEpistandardMorphisms}, the epistandard morphism $\mu_{va}$ is strongly quasiperiodic. Thus $\mu_{\breve u}(a\bt)$ is quasiperiodic. Consequently $\mu_{\breve u}$ is strongly quasiperiodic on $\cA$.

\medskip

To end, we prove that if $\breve u$ verifies none of the conditions $i)$--$iii)$, then there exists (at least) one word $\bt$ such that $\mu_{\breve u}(\bt)$ is not quasiperiodic (and so $\mu_{\breve u}$ is not strongly quasiperiodic). This is immediate if $\breve u=\varepsilon$.

\begin{itemize}

\item
Let us first consider the case where $\breve u$ ends with an $L$-spinned letter $a$, that is $\breve u=\breve w a$ for some spinned word $\breve w$. Since $|\cA| \geq 2$, there exist $m \geq 2$ pairwise different letters $a_1 = a,\, a_2,\, \ldots ,\, a_m$ such that $\cA = \{a_1, a_2, \ldots, a_m\}$. Let $\bt$ be the episturmian word with normalized directive word $(\bar{a}_2 \cdots \bar{a}_m a)^\omega$. Since $\breve w a$ is normalized, $\breve w a (\bar{a}_2 \cdots \bar{a}_m a)^\omega$ is also normalized. Moreover since $\breve u$ does not verify condition $i)$, the word $\breve w a (\bar{a}_2 \cdots \bar{a}_m a)^\omega$ cannot be decomposed into the form $\breve w' b v_1 \bar b v_2 \cdots \bar b v_k \breve y' \cdots$ where $b$ is an $L$-spinned letter, $\breve w'$ is a spinned word, $\breve y'$ is a spinned version of an $L$-spinned word $y'$, and $v_1, \ldots, v_k$ ($k \geq 0$) are $b$-free $L$-spinned words such that $\Alph(bv')=\Alph(bv'y')$. By Remark~\ref{R:decomp} and Theorem~\ref{T:normalized-quasi-characterization}, the word $\mu_{\breve u}(\bt)$ is not quasiperiodic.

\item
Now we consider the case when $\breve u$ ends with an $R$-spinned letter, that is $\breve u = \breve w \bar v$ for a non-empty $L$-spinned word $v$ and a spinned word $\breve w$ such that $\breve w=\varepsilon$ or $\breve w$ ends with an $L$-spinned letter. Two cases can hold:

\begin{description}
\item{Case 1:} $\cA \neq \Alph(v)$.

Let $a$ be a letter in $\cA \setminus \Alph(v)$, and let $b$ be any other letter (remember $|\cA|\geq 2$).
Let $\bt$ be the episturmian word with normalized directive word $(a \bar b)^\omega$ and let $\bt' = \mu_{\breve u}(\bt)$. Then $\bt'$ is directed by $\breve \bu'=\breve w \bar v (a \bar b)^\omega$ and, since $\breve w \bar v$ is normalized and $\bar a \not \in \Alph(\bar v)$, $\breve \bu'$ is normalized. Moreover since $\breve w \bar v$ does not verify $i)$, $\breve \bu'$ cannot be decomposed into the form $\breve w' c v_1 \bar c v_2 \cdots \bar c v_k \breve y' \cdots$ with 
an $L$-spinned letter $c$, a spinned word $\breve w'$, a spinned version $\breve y'$ of an $L$-spinned word $y'$ and some $c$-free $L$-spinned words $v_1, \ldots, v_k$ ($k \geq 0$) such that $\Alph(cv_k)= \Alph(cv_ky')$.
By Remark~\ref{R:decomp} and Theorem~\ref{T:normalized-quasi-characterization}, the word $\mu_{\breve u}(\bt)$ is not quasiperiodic.

\item{Case 2:} $\cA = \Alph(v)$.

Since $\breve u$ does not verify $ii)$, there exists a letter $a$ and $a$-free $L$-spinned words $v_0, \ldots, v_k \in \cA^*$ such that $\breve u = \breve w \bar v_0 \bar a \bar v_1 \cdots \bar a \bar v_k$ and 
for all $i$, $1 \leq i \leq k$, $\Alph(v_i) \neq \cA\setminus \{a\}$.
Moreover since $\breve u$ does not verify $iii)$, then either $v_0 \neq \varepsilon$, or $\breve w$ cannot be written in the form $\breve w=\breve w' v'$ for a spinned word $\breve w$ and an $L$-spinned word $v'$ such that $\Alph(v')=\cA \setminus \{a\}$.

Since $\cA$ contains at least three letters, there exist $m\geq 3$ pairwise different letters $a_1 = a$, $a_2$, \ldots, $a_m$ such that $\cA=\{a_1, a_2, \ldots, a_m\}$. 

Let $\bt$ be the episturmian word with normalized directive word $(a a_2 \bar a_3 \cdots \bar a_m)^\omega$ and let $\bt'=\mu_{\breve u}(\bt)$.
Then $\bt'$ is directed by $\breve w \bar v_0 \bar a \bar v_1 \cdots \bar a \bar v_k a a_2 \bar a_3 \cdots \bar a_m (a a_2 \cdots \bar a_m)^\omega  \equiv \breve w \bar v_0 a v_1 \bar a \cdots \bar a v_k \bar a a_2 \bar a_3 \cdots \bar a_m (a a_2 \bar{a}_3\cdots \bar a_m)^\omega$.
Since $\breve w$ ends with an $L$-spinned letter and each $v_i$ ($0 \leq i \leq k$) is $a$-free, this word is normalized.

Since $\breve w \bar v$ does not verify $i)$ and from the previous observations, the word $\breve w \bar v_0 a v_1 \bar a \cdots \bar a v_k \bar a a_2 \bar a_3 \cdots \bar a_m (a a_2 \cdots \bar a_m)^\omega$ cannot be decomposed into the form $\breve w' b v_1' \bar b v_2' \cdots \bar b v_k' \breve \by' \cdots$ with 
an $L$-spinned letter $b$, a spinned word $\breve w'$, a spinned version $\breve \by'$ of an $L$-spinned word $\by'$ and some $b$-free $L$-spinned words $v_1', \ldots, v_k'$ ($k \geq 0$) such that with $\Alph(bv_k')=\Alph(bv_k'\breve \by')$.
By Remark~\ref{R:decomp} and Theorem~\ref{T:normalized-quasi-characterization}, the word $\mu_{\breve u}(\bt)$ is not quasiperiodic.
\end{description}
\end{itemize}
\end{proof}

\normalsize

\section{Episturmian Lyndon words} \label{S:episturmianLyndon}

Theorem~\ref{T:quasi-characterization} provides a characterization of quasiperiodic episturmian words. In the binary case, it was proved in \cite{fLgR07quas} that a Sturmian word is quasiperiodic if and only if it is not an infinite Lyndon word. A natural question to ask is then: ``does this result still hold for episturmian words on a larger alphabet?'' By a result in \cite{fLgR07quas}, one can see that any infinite Lyndon word is non-quasiperiodic. In this section, we show that there is a much wider class of episturmian words that are non-quasiperiodic, besides those that are infinite Lyndon words. This follows from our characterization of episturmian Lyndon words (Theorem~\ref{T:Lyndon-episturmian}, to follow).

Let us first recall the notion of lexicographic order and the definition of Lyndon words (see \cite{mL83comb} for instance).

Suppose the alphabet $\cA$ is totally ordered by the relation $<$. Then we 
can totally order $\cA^*$ by the \emph{lexicographic order} $\leq$ 
defined as follows. Given two words $u$, $v \in \cA^+$, we have $u
\leq v$ if and only if either $u$ is a prefix of $v$ or $u =
xau^\prime$ and $v = xbv^\prime$, for some $x$, $u^\prime$,
$v^\prime \in \cAstar$ and letters $a$, $b$ with $a < b$. This is
the usual alphabetic ordering in a dictionary. We write $u < v$ when $u\leq v$ and $u\ne v$, in which case we say that $u$ is (strictly) \emph{lexicographically smaller} than $v$. The notion of lexicographic order naturally extends to infinite words in $\cAw$. We denote by $\min(\cA)$ the smallest letter with respect to the lexicographic order.

A non-empty finite word $w$ over $\cA$ is a {\em Lyndon word} if it is lexicographically smaller than all of its proper suffixes for the given order $<$ on $\cA$. Equivalently, $w$ is the lexicographically smallest primitive word in its conjugacy class; that is, $w < vu$ for all non-empty words $u$, $v$ such that $w = uv$. The first of these definitions extends to infinite words: an infinite word over $\cA$ is an {\em infinite Lyndon word} if and only if it is (strictly) lexicographically smaller than all of its proper suffixes for the given order on $\cA$. That is, a finite or infinite word $w$ is a Lyndon word if and only if $w < \TT^i(w)$ for all $i > 0$.

In this section, we assume that $|\cA| > 1$ since on a $1$-letter alphabet there are no infinite Lyndon words. Also note that an infinite Lyndon word cannot be periodic. Therefore we consider only aperiodic episturmian words (i.e., those with $|\Ult(\Delta)| > 1$).

\subsection{A complete characterization} \label{SS:Lyndon-characterization}

In this section, generalizing previous results in \cite{fLgR07quas} (Sturmian case) and \cite{aG07orde} (Arnoux-Rauzy sequences or strict episturmian words), we prove:

\begin{thrm} \label{T:Lyndon-episturmian}
Let $\cA = \{a_1, \ldots, a_m\}$ be an alphabet ordered by $a_1 < a_2 < \cdots < a_m$ and, for $1 \leq i \leq m$, let $\cB_i = \{a_i, \ldots, a_m\}$. 

An episturmian word $\bw$ is an infinite Lyndon word if and only if there exists an integer $j$ such that $1 \leq j <m$ and the (normalized) directive word of $\bw$ belongs to:
$$(\bar\cB_2^*a_1)^*\cdots 
(\bar\cB_j^*a_{j-1})^*(\bar\cB_{j+1}^*a_j)^*(\bar\cB_{j+1}^+\{a_j\}^+)^\omega.$$
\end{thrm}

\begin{note} In the above theorem, we have put the word \emph{normalized} between brackets since one can easily verify from Theorem~\ref{T:uniqueDirective} that a spinned infinite word of the given form is the unique directive word of exactly one episturmian  word. 
\end{note}

\begin{xmpl}
Let $\cA=\{a,b,c,d\}$. Then the word $(\bar b \bar c a)(\bar d \bar c b)^2 (\bar d c c)^\omega$ directs a Lyndon episturmian word, so does $aa(\bar d c)^\omega$, but $\bar c a \bar b a \bar d c d^\omega$ does not (this spinned word directs a periodic word).
\end{xmpl}

\begin{rmrk} The ``if and only if'' condition can be reformulated as follows: 
\begin{quote} {\it 
The (normalized) directive word of $\bw$ takes the form $v_1\cdots v_{j} \by$
where:
\begin{itemize}
\item for $1 \leq k \leq j$, $v_k$ is a spinned word in $(\bar{\cB}_{k+1}^*a_k)^*$;
\item $\by$ is a spinned infinite word belonging to $(\bar{\cB}_{j+1}^+\{a_j\}^+)^\omega$.
\end{itemize}}
\end{quote}
\end{rmrk}

\begin{rmrk} \label{R:uniqueDirective-Lyndon}
Theorems \ref{T:uniqueDirective} and \ref{T:Lyndon-episturmian} show that any episturmian Lyndon word has a unique spinned directive word, but the converse is not true. Certainly, there exist episturmian words with a unique directive word which are not infinite Lyndon words. For example, the regular wavy word $(a\bar b c)^\omega$ is the unique directive word of the strict episturmian word: 
\[
\lim_{n\rightarrow\infty}{\mu_{a\bar b c}^n(a)} = acabaabacabacabaabaca\cdots 
\]
which is clearly not an infinite Lyndon word by Theorem~\ref{T:Lyndon-episturmian} and also by the fact that $acabaaw$ is not a Lyndon word for any order on $\{a,b,c\}$ and for any word $w$. 
\end{rmrk}

In order to prove the Theorem~\ref{T:Lyndon-episturmian}, we recall two useful results and state a new one (Lemma~\ref{L:Lyndon-properties}).

\begin{lmm}{\rm \cite{Mel2000}}
\label{defMelancon}
An infinite word is a Lyndon word if and only if it has infinitely many different Lyndon words as prefixes. 
\end{lmm}

A morphism $f$ is said to preserve finite (resp.~infinite) Lyndon words if for each finite (resp.~infinite) Lyndon word $w$, $f(w)$ is a finite (resp.~infinite) Lyndon word. For episturmian morphisms, we have:

\begin{prpstn} {\em \cite{gR03lynd, gR07onmo}} \label{P:Lyndon-morphisms}
Let $\cA = \{a_1, \ldots, a_m\}$ be an alphabet ordered by $a_1 < a_2 < \cdots < a_m$. Then the following assertions are equivalent for an episturmian morphism:
\begin{itemize}
\item $f$ preserves finite Lyndon words;
\item $f$ preserves infinite Lyndon words;
\item $f \in (\mathcal{R}_{\{a_2, \ldots, a_m\}}^*L_{a_1})^*\{R_{a_m}\}^*$.
\end{itemize}
\end{prpstn}

Now we prove a lemma concerning the action of morphisms in $\cL_{\cA}\cup\cR_{\cA}$ on infinite Lyndon words.

\begin{lmm} \label{L:Lyndon-properties}
Suppose $\bw$ is an infinite word over an ordered alphabet $\cA$ and let $a$, $b \in \cA$ with $a < b$. Then, the following properties hold.
\begin{enumerate}
\item[i)] $L_bfL_a(\bw)$ is not an infinite Lyndon word for any non-erasing morphism $f$.
\item[ii)] $R_afL_b(\bw)$ is not an infinite Lyndon word for any morphism $f$ in $\cR_{\cA}^*$.
\item[iii)] If $\bw$ is recurrent, then $R_xfL_x(\bw)$ is not an infinite Lyndon word for any letter $x$ and  morphism $f$ in $\cal{R}_{\cA}^*$.
\end{enumerate}
\end{lmm}
\begin{proof} $i)$ The infinite word $L_bfL_a(\bw)$ starts with $b$ and contains an occurrence of the letter $a$; thus, since $a < b$, it cannot be an infinite Lyndon word. \medskip

\noindent $ii)$ As $f \in \cR_{\cA}^*$, the infinite word $R_afL_b(\bw)$ starts with $b$ and contains an occurrence of the letter $a$; thus, since  $a < b$, it cannot be an infinite Lyndon word. \medskip

\noindent $iii)$ To be an infinite Lyndon word, $R_xfL_x(\bw)$ must be aperiodic, in which case a letter different from $x$ occurs in it; in particular this letter occurs in $fL_x(\bw)$. Moreover, as $f\in \cR_{\cA}^*$, the infinite word $fL_x(\bw)$ begins with $x$ and hence with a prefix $x^ny$ for some integer $n \geq 1$ and a letter $y \neq x$. The recurrence of the infinite  word $\bw$ implies the recurrence of $fL_x(\bw)$, and so $fL_x(\bw)$ contains a factor $zx^{n+r}y$ for some letter $z \neq x$ and  integer $r \geq 0$. Now $R_xfL_x(\bw)$ begins with $x^ny$ and contains $zx^{n+r+1}y$, and so contains $x^{n+1}$. To be an infinite Lyndon word, it needs $x = \min(\Alph(R_xfL_x(\bw)))$, but then $x^{n+1} < x^n y$. Thus $R_xfL_x(\bw)$ is not an infinite Lyndon word. 
\end{proof}

Lastly, we need an important easy fact:

\begin{fact}\cite{gR07onmo}\label{F:episturmian_preserve_order}
Any morphism $f$ in $\cL_{\cA}\cup\cR_{\cA}$ preserves the lexicographic order for infinite words. More precisely, for any infinite words $\bw$ and $\bw'$, $\bw < \bw'$ if and only if $f(\bw) < f(\bw')$.
\end{fact}

A consequence of the above fact is that for any word $\bw$ and for any morphism $f$ in $\cL_{\cA}\cup\cR_{\cA}$, if $f(\bw)$ is a Lyndon word then necessarily $\bw$ is also a Lyndon word.

\begin{proof}[Proof of Theorem $\ref{T:Lyndon-episturmian}$] Assume that $\Delta$ is the normalized directive word of a Lyndon episturmian word. Then $\Delta$ contains no factor of the form $\bar{x}\bar vx$ for any letter $x$ and $v \in \cA^*$. By 
Lemma~\ref{L:Lyndon-properties}, it does not contain any factor of the form $bva$ or $\bar{a}\bar vb$ with $v \in \cA^*$ and  $a <b$. Thus $\Delta$ takes the form given in the statement of the theorem. Indeed by item $i)$ of Lemma~\ref{L:Lyndon-properties} and by Fact~\ref{F:episturmian_preserve_order}, only one letter (namely $a_j$) can have all spins ultimately $L$. Since a Lyndon word is not periodic, at least one other letter in $\cA$ should occur infinitely often. By items $ii)$-$iii)$ of Lemma~\ref{L:Lyndon-properties}, such a  letter should belong to $\bar{\cal B}_{j+1}$. Moreover, the sequence of letters with spin $L$ must be order-increasing and items $ii)$--$iii)$ of Lemma~\ref{L:Lyndon-properties} determine the conditions on letters with spin~$R$. 
\medskip

Conversely, suppose that the (normalized) directive word $\Delta$ of the episturmian word $\bw$ takes the form given in the statement of the theorem. Write $\Delta = v_1 \cdots v_{j}\by$ where, for $1 \leq k \leq j$, $v_k$ is a spinned word in $(\bar\cB_{k+1}^*a_k)^*$ and $\by$ is a spinned infinite word belonging to $(\bar\cB_{j+1}^+\{a_j\}^+)^\omega$. Then, because of the recurrence of the letter $a_j$ and of at least one other letter in $\cB_{j+1}$ in $\by$, there exists a sequence of spinned words $(v_{n})_{n \geq j}$, with each $v_{n}$ in $\bar\cB_{j+1}^+\{a_j\}^+$, such that $\by = \prod_{n \geq j} v_{n}$. Now, for each $k \geq 1$, $\mu_{v_k}$ is a Lyndon morphism on $\cB_k$ by Proposition~\ref{P:Lyndon-morphisms}. Hence, for each $k \geq 1$, the word $\mu_{v_1\cdots v_k}(a_j)$ is a Lyndon word. From $\bw = \lim_{k \to \infty} \mu_{v_1\cdots v_k}(a_j)$, we deduce from Lemma~\ref{defMelancon} that the episturmian word $\bw$ is an infinite Lyndon word.
\end{proof}

\subsection{Strict episturmian Lyndon words}

Let us recall from \cite{jJgP02epis} that an epistandard word $\bs$, or any episturmian word in the subshift of $\bs$, is said to be $\cA$-strict if its $L$-spinned directive word $\Delta$ verifies Ult$(\Delta) = \cA$. For these words, also called Arnoux-Rauzy sequences \cite{pAgR91repr}, Theorem~\ref{T:Lyndon-episturmian} gives:

\begin{crllr} \label{Cor:as&Lyndon-2}
Let $\cA = \{a_1, \ldots, a_m\}$ be an alphabet ordered by $a_1 < a_2 < \cdots < a_m$. An $\cA$-strict episturmian word $\bw$ is an infinite Lyndon word if and only if the (normalized) directive word of $\bw$ belongs to $\{a_1,\bar a_2, \ldots, \bar a_m\}^\omega.$ 
\end{crllr}

This can be reformulated as a generalization of Proposition 6.4 in \cite{fLgR07quas}:

\begin{crllr} {\em \cite{aG07orde}} \label{Cor:decomposable2} An $\cA$-strict episturmian word $\bt$ is an infinite Lyndon word if and only if it can be infinitely decomposed over the set of morphisms $\{L_a, R_x \mid x \in \cA\setminus\{a\}\}$ where $a = \min(\cA)$ for the given order on~$\cA$. 
\end{crllr}

The above result also follows from the following generalization of a result on Sturmian words given by Borel and Laubie~\cite{jpBfL93quel} (see also \cite{gR07conj}).

\begin{thrm} 
\label{T:Lyndon}
An $\cA$-strict episturmian word $\bt$ is an infinite Lyndon word if and only if $\bt = a\bs$ where $a = \min(\cA)$ for the given order on $\cA$ and $\bs$ is an (aperiodic) $\cA$-strict epistandard word. Moreover, if $\Delta$ is the $L$-spinned directive word of $\bs$, then $\bt = a\bs$ is the unique episturmian word in the subshift of $\bs$  directed by the spinned version of $\Delta$ having all spins $R$, except when $x_i = a$.  
\end{thrm}

The proof of the above theorem requires the following result that is essentially Theorem~3.17 from~\cite{jJgP02epis}, apart from the fact that $a\bs$ is in the subshift of $\bs$, which follows from  Fact~\ref{F:episturmian}.

\begin{thrm} \label{T:3.17} 
Suppose $\bs$ is an epistandard word directed by $\Delta = x_1x_2x_3 \cdots$ and let $a$ be a letter. Then $a\bs$ is an episturmian word if and only if $a \in \eUlt(\Delta)$, in which case $a\bs$ is the unique episturmian word in the subshift of $\bs$ 
directed by the spinned version of $\Delta$ having all spins $R$, except when $x_i = a$.
\end{thrm} 

\begin{proof}[Proof of Theorem~$\ref{T:Lyndon}$]
Let $\cA = \{a_1, a_2, \ldots, a_m\}$ with $a_1 < a_2 < \cdots < a_m$. By Corollary~\ref{Cor:as&Lyndon-2}, an $\cA$-strict episturmian word $\bt$ is an infinite Lyndon word if and only if the (normalized) directive word $\breve\Delta$ of $\bt$ belongs to $\{a_1, \bar a_2, \ldots, \bar a_m\}^\omega$, i.e., if and only if $\bt = a_1\bs$ where $\bs$ is the unique epistandard word directed by the $L$-spinned version of $\breve\Delta$, by Theorem~\ref{T:3.17}. 
\end{proof}

\begin{note} If $\bs$ is an epistandard word over $\cA$, then $a\bs$ is an infinite Lyndon word for any order such that $a = \min(\cA)$.
\end{note}

Let us point out that  completely different proofs of Corollary~\ref{Cor:as&Lyndon-2} and Theorem~\ref{T:Lyndon}, using a characterization of episturmian words via lexicographic orderings, were given in \cite{aG07orde} by the first author. A refinement of one of the main  results in \cite{aGjJgP06char} is also given in \cite{aG07orde}.

\medskip

To end, let us observe that, contrary to the fact that there exists $|\cA|!$ possible orders of a finite alphabet $\cA$, Theorem~\ref{T:Lyndon} shows that there exist exactly $|\cA|$ infinite Lyndon words in the subshift of a given $\cA$-strict epistandard word $\bs$, when $|\cA| >1$ (since there are no Lyndon words when $|\cA| = 1$). That is, for any order with $\min(\cA) = a$, the subshift of $\bs$ contains a unique infinite Lyndon word beginning with $a$, namely $a\bs$.  

\begin{xmpl}
With $\Delta = (abcd)^\omega$, the spinned versions $(a\bar b \bar c \bar d)^\omega$, $(\bar a b \bar c \bar d)^\omega$, $(\bar a \bar b c \bar d)^\omega$, $(\bar a \bar b \bar c d)^\omega$, $(\bar a \bar b c d)^\omega$, $(\bar a b \bar c d)^\omega$, $(\bar a b c \bar d)^\omega$ and their opposites direct non-quasiperiodic episturmian words in the subshift of the  $4$-bonacci word $\bz$.  Only the first four of these words direct Lyndon episturmian words: $a\bz$, $b\bz$, $c\bz$, $d\bz$, respectively.
\end{xmpl}

\section{Concluding remarks} \label{conclusion}

In \cite{Mon05}, Monteil proved that any Sturmian subshift contains a {\em multi-scale quasiperiodic word}, i.e., an infinite word having infinitely many quasiperiods. A shorter proof of this fact was provided in \cite{fLgR07quas}. This can be easily extended to episturmian words. Certainly, by Fact~\ref{F:subshift}, any episturmian subshift contains at most two epistandard words (one in the aperiodic case and two in the periodic case) and any epistandard word has infinitely many quasiperiods (by Theorem~\ref{T:epistandard-quasiperiods}).

Actually the characterization of quasiperiodic Sturmian words in \cite{fLgR07quas} shows that in any Sturmian subshift there are only two non-quasiperiodic Sturmian words and all other (Sturmian) words in the subshift have infinitely many quasiperiods. It is easy to see that the same result does not hold for episturmian words defined over an alphabet containing more than two letters. For instance, any episturmian word having a spinned directive word in $\{ab\bar c, a\bar b\bar c\}^\omega$ is non-quasiperiodic: all of these non-quasiperiodic episturmian words belong to the subshift of the Tribonacci word $\br$, directed by $(abc)^\omega$. Moreover, one can verify (using Theorem~\ref{T:directSame}) that the quasiperiodic episturmian word $\bt$ directed by $(abc)^n(ab\bar c)(a\bar b \bar c)^\omega$ for some $n\geq 1$ (which is in the subshift of $\br$) has exactly $n+1$ directive words:
\[
 (abc)^iab \bar c(\bar a \bar b c)^{n-i}(a \bar b \bar c)^\omega, \quad 0 \leq i \leq n.
 \]
Hence it is clear from Theorem~\ref{T:episturmianQuasiperiods} that $\bt$ has only finitely many quasiperiods. (See also Examples~\ref{ex:1}--\ref{ex:3}.)

\bigskip

\noindent
\textbf{Acknowledgments:} The authors would like to thank the two anonymous referees for their suggestions to improve the paper. In particular, one referee indicated the simple proof of Theorem~\ref{T:epistandard-quasiperiods} and the other noticed the interest of Fact~\ref{F:from_referee}. Many thanks also to D.~Krieger who suggested to the third author the notion of ultimate quasiperiodicity and observed the fact that all Sturmian words are ultimately quasiperiodic. The interest of this notion becomes evident when one referee underlined that the term \textit{quasi-factor} introduced in \cite{aG07orde} was not completely satisfactory.

\footnotesize
\bibliographystyle{plain}
\bibliography{GLR}

\end{document}